%% file: main.tex
\documentclass[11pt]{article} 
\usepackage[utf8]{inputenc}
\usepackage[T1]{fontenc}
\usepackage{tikz}
\usepackage[a4paper]{geometry} 
\usepackage{microtype} 
\usepackage{lmodern}
\usepackage{xcolor}
\usepackage{bigints}
\usepackage{setspace}
\usepackage{amssymb}
\usepackage{amsmath}
\usepackage{amsthm}
\usepackage{mathtools}
\usepackage{graphicx}
\usepackage{subfig}
\usepackage{verbatim}
\usepackage[numbers]{natbib}
\usepackage{bbm, dsfont}
\usepackage{relsize}
\usepackage{enumerate}
\usepackage[inline]{enumitem}   
\usepackage{booktabs}

\usepackage{caption}

\makeatletter
\newcommand{\inlineitem}[1][]{%
\ifnum\enit@type=\tw@
    {\descriptionlabel{#1}}
  \hspace{\labelsep}%
\else
  \ifnum\enit@type=\z@
       \refstepcounter{\@listctr}\fi
    \quad\@itemlabel\hspace{\labelsep}%
\fi}

\usepackage{MnSymbol}
\usepackage{algpseudocode}
\usepackage{algorithm}
\usepackage{hyperref} 
\usepackage{wasysym}

\newtheorem{theorem}{Theorem}[section]
\newtheorem{corollary}[theorem]{Corollary}
\newtheorem{lemma}[theorem]{Lemma}
\newtheorem{proposition}[theorem]{Proposition}

\theoremstyle{definition}
\newtheorem{definition}[theorem]{Definition}
\newtheorem{example}[theorem]{Example}

\theoremstyle{remark}
\newtheorem{remark}[theorem]{Remark}

\numberwithin{equation}{section}

\input{newcommands}

\title{Multivariate generalized Pareto distributions along extreme directions}

\author{Anas Mourahib\thanks{UCLouvain, LIDAM/ISBA, anas.mourahib@uclouvain.be. Corresponding author. Supported by the \emph{Fonds de la Recherche Scientifique – FNRS} (grant number T.0203.21).} \and Anna Kiriliouk\thanks{UNamur, NaXys, anna.kiriliouk@unamur.be.} \and Johan Segers\thanks{UCLouvain, LIDAM/ISBA, johan.segers@uclouvain.be.}}
\date{\today}

\begin{document}

\maketitle

\begin{abstract}
When modeling a vector of risk variables, extreme scenarios are often of special interest. The peaks-over-thresholds method hinges on the notion that, asymptotically, the excesses over a vector of high thresholds follow a multivariate generalized Pareto distribution. However, existing literature has primarily concentrated on the setting when all risk variables are always large simultaneously. In reality, this assumption is often not met, especially in high dimensions.

In response to this limitation, we study scenarios where distinct groups of risk variables may exhibit joint extremes while others do not. These discernible groups are derived from the angular measure inherent in the corresponding max-stable distribution, whence the term \emph{extreme direction}. We explore such extreme directions within the framework of multivariate generalized Pareto distributions, with a focus on their probability density functions in relation to an appropriate dominating measure.

Furthermore, we provide a stochastic construction that allows any prespecified set of risk groups to constitute the distribution's extreme directions. This construction takes the form of a smoothed max-linear model and accommodates the full spectrum of conceivable max-stable dependence structures. Additionally, we introduce a generic simulation algorithm tailored for multivariate generalized Pareto distributions, offering specific implementations for extensions of the logistic and Hüsler–Reiss families capable of carrying arbitrary extreme directions.
\end{abstract}


 
\section{Introduction}
The statistical description of extremes is a long-standing subject of research, with the ultimate goal of extrapolating beyond the range of the data and estimating probabilities of events more extreme than those observed in the past. The foundations of univariate extreme value theory were established many decades ago through the derivation of the asymptotic distributions of suitably normalized block maxima and 
excesses over high thresholds.
Modeling multivariate extreme events is a bigger challenge because, contrary to the univariate setting, the class of multivariate extreme-value distributions is infinite-dimensional. Over the last thirty years, numerous parametric extreme-value models have been proposed with the goal of capturing the tail dependence structure of a random $d$-dimensional vector $\bX = (X_1,\ldots,X_d)$; see, e.g., \citet[Chapter 9]{Be04} for an early overview or \citet{cooley2010pairwise}, \citet{ballani2011construction} and \citet{opitz2013extremal} for more recent examples. However, many such models are limited to the setting where all components of $\bX$ can become large simultaneously. While this may be sufficient when the dimension $d$ of $\bX$ is low, more often than not, one encounters the setting where some variables 
may be large while at the same time, the others are of smaller order. 
In finance, for example, characterizing groups of stocks within a portfolio that have the potential of suffering major losses simultaneously is of great importance. Managing flood risk requires identifying and modeling groups of gauging stations where extreme water discharges may occur simultaneously. 

Groups of variables that may become large without the others being so---in an asymptotic sense to be made precise below---will be called \textit{extreme directions}. More specifically, if $J \subseteq \{1,\ldots,d\}$ is such that the components $(X_j, j \in J)$ can be large simultaneously while the other components $(X_j, j \in \{1,\ldots,d\} \setminus J)$ are small, then $J$ defines an extreme direction. The terminology is in line with a characterization of the support of the angular measure of the multivariate extreme-value distribution to which $\bX$ is attracted. In the literature, such groups are also referred to as extreme clusters \citep{chiapino2020}, extreme features \citep{jalalzai2021feature}, or concomitant extremes \citep{fomichov2023spherical}.
Many parametric extreme-value families are limited to the single extreme direction $J = \{1,\ldots,d\}$. The overarching goal of this paper is to explore the case when other extreme directions than $\{1,\ldots,d\}$ exist, and this in terms of several well-known objects in extreme-value theory, and multivariate generalized Pareto (mgp) distributions in particular.

In the literature, tail dependence is described mathematically through various concepts, such as the exponent measure, the angular measure, and the stable tail dependence function, among others. When the focus is on modeling multivariate threshold exceedances, the mgp distribution is a natural candidate \citep{TR06}. In the univariate case, peaks-over-threshold modeling with the generalized Pareto distribution is well-established \citep{davison1990models}, while in the multivariate set-up, statistical modeling is quite recent \citep{R18nr2,A16,ho2019simple,engelke2020graphical,hentschel2022statistical}.
\amch{Asymptotically, the class of mgp distributions represents proper multivariate probability distributions on a region where at least one variable is large. Moreover, parametric mgp models can be constructed quite easily via the \emph{generators} proposed in \citet{A16}.}

Current theory for mgp distributions does not cover extreme directions other than $J = \{1,\ldots,d\}$, a gap we aim to fill. From their generators, we derive expressions for mgp densities with respect to an appropriate dominating measure. 
\amch{Furthermore, we propose a stochastic construction based on smoothed max-linear models, wherein the strong dependence within factors inherent in max-linear models is substituted with a more lenient dependency framework. Our construction is designed to accommodate all mgp models, regardless of their extreme directions.}

Learning extreme directions from data has become an active area of research. 
\citet{Go17, goix2017JMVA} focus on the identification of extreme directions, providing a sparse representation of the support of the angular measure. 
\citet{AM16} and \citet{chiapino2019identifying} study an alternative algorithm with better performance when the number of extreme directions is high. \citet{jalalzai2021feature} introduce an approach based on empirical risk minimization. \citet{chiapino2020} propose a Dirichlet mixture model for the angular measure of a heavy-tailed distribution whose support has been identified and apply it to anomaly detection. \citet{simpson2020determining} exploit the concept of hidden regular variation to identify extreme directions. Finally, \citet{meyer2021sparse,meyer2020multivariate} frame extreme directions within what they call sparse regular variation. 

The study of extreme directions is related to that of clustering and dimension reduction methods for multivariate extremes.
For example, \citet{chautru2015dimension} and \citet{janssen2020k} identify extreme directions via spherical $k$-means clustering of the angles $\bX/ \|\bX\|$ provided $\| \bX \|$ is large. \citet{fomichov2023spherical} provide theoretical results in favor of this approach and extend it to a spherical $k$-principal components algorithm. Clustering angles is also considered in \citet{medina2021}, based on a random $k$-nearest neighbors graph. Finally, \citet{cooley2019decompositions} and \citet{drees2021principal} propose to perform principal component analysis of the angles $\bX/ \| \bX \|$. These and other methods are surveyed in \citet{engelke2021sparse}.

\amch{Our contribution differs from those above from several perspectives. First, our focus is on extreme directions in the context of multivariate threshold exceedances. Second, our aim is not to learn extreme directions, but rather to introduce parametric mgp models capable of generating extreme directions other than $\cbr{1,\ldots,d}$. For multivariate extreme-value distributions, the max-linear factor model \citep{Ei12} and the asymmetric logistic model \citep{T1990} already allow for this possibility. In the recent literature, only \citet{simpson2020determining} and \citet{chiapino2020} propose mixture models for multivariate extreme-value distributions. The power of our model construction resides in the fact that it is capable of representing mgp distributions with general extreme directions. Moreover, for specific parametric families, our model is straightforward to simulate from.}

The paper is organized as follows. Section~\ref{sec:background} recalls essential background on multivariate extreme-value theory. Section~\ref{sec:Directions} defines and characterizes extreme directions in terms of the angular measure, the exponent measure, and the mgp distribution. We also compute the density of an mgp distribution with arbitrary extreme directions in terms of its generators, extending results in \citet{A16}. A model construction tool based on a mixture of mgp generators is presented in Section~\ref{sec:sparsemgpds}. Apart from enjoying a convenient interpretation as a smoothed max-linear model, the construction is able to reproduce any mgp distribution. For the sake of illustration, we work out two examples based on the logistic and the \HR{} families.
Moreover, the construction underlies a simulation algorithm for mgp distributions with non-standard extreme directions (Section~\ref{sec:simulation}). As a by-product, 
\amtwo{we revisit known simulation algorithms in \citet{dombry2016exact} and \citet{engelke2020graphical} that can be used to simulate mgp random vectors}. Section~\ref{sec:conclusion} brings the paper to a close by outlining several avenues for future research. All proofs are deferred to the appendices.

\section{Background}
\label{sec:background}

\paragraph{Notation.}  
\amch{Let $d$ be a positive integer and write $D = \cbr{1,\ldots,d}$. Throughout, bold symbols will refer to multivariate quantities. Let $\bzero$ and $\bone$ denote vectors of zeroes and ones, respectively, of a dimension clear from the context. A $d$-dimensional vector $\bo{x}$ will be written as $\bx=(x_1,\ldots,x_d)$ and for a non-empty set $J \subseteq \cbr{1,\ldots,d}$,  write $\bx_J=(x_j)_{j \in J}$. Mathematical operations on vectors such as addition, multiplication and comparison are considered component-wise apart from $\bo{a} \nleqslant \boldsymbol{b}$ which denotes $a_j>b_{j}$ for at least one $j \in D$.}

We also fix notation for some sets that will be used throughout the paper. For two vectors $\bo{a}, \; \bo{b}$, the set $[\bo{a},\bo{b}]$ will denote the Cartesian product $\prod_{j=1}^{d}[a_j,b_j]$ and so forth for $(\bo{a},\bo{b}), \; [\bo{a},\bo{b}), \; (\bo{a},\bo{b}]$. Write $\EE = [0,\infty)^d\setminus \cbr{\bzero}$. Let $\borel(F)$ denote the Borel sets over a topological space $F$. For a general norm $\|\,\cdot\,\|$ on $\Rd$, let $\mathbb{S}=\cbr{\bx \in  \EE : \; \|\bx\|=1}$ denote the intersection of the unit sphere with the positive orthant. For a non-empty set $J \subseteq D$, let 
\begin{equation}
\label{eq:AJ}
	\mathbb{A}_{J}:=\cbr{ \bx \in [-\infty,\infty)^d: \; x_j>-\infty \text{ iff } j \in J },
\end{equation}
and let $\pi_{J} : \mathbb{A}_J \to \reals^{J}$ be the projection $\bx \mapsto \bx_{J}$ \amtwo{and with $\reals^J$ the set of real vectors $(x_j)_{j \in J}$ indexed by $J$. Clearly, $\reals^J$ is isomorphic to $\reals^{|J|}$, with $\lvert J \rvert$ the number of elements of $J$, but it will be convenient to keep track of the index set explicitly, not only its cardinality.}

The arrow $\dto$ denotes convergence in distribution (weak convergence). If $\mu$ is a measure, then a Borel set $B$ is $\mu$-continuous if $\mu(\partial B)=0$, with $\partial B$ the topological boundary of $B$. The distribution or law of a random vector $\bX$ is denoted by $\law(\bX)$. Finally, the Lebesgue measure on $\reals^m$ is denoted by $\lambda_m$. 
\ak{At the end: check if all references are still up to date and check if all numbered equations are actually referred to (if not, remove their label).}

\subsection{Multivariate extreme value distributions} \label{09:12:15:48}

Consider a random vector $\bxi$ on $\Rd$ with joint cumulative distribution function (cdf) $F$ and margins $F_1,\ldots,F_d$. We say that $F$ is in the domain of attraction of a multivariate extreme-value (mev) distribution $G$ and we write $F \in \DA(G)$ if there exist sequences $\boldsymbol{a}_{n} \in (0,\infty)^d$ and $\boldsymbol{b}_{n} \in \Rd$ such that
\[
    \lim_{n \to \infty} F^{n}(\boldsymbol{a}_{n}\bx+\boldsymbol{b}_{n}) = G(\bx), 
\]
for every point $\bx \in \reals^d$ where $G$ is continuous. Equivalently, if $\bxi_{1},\ldots,\bxi_{n}$ are i.i.d.\ random vectors with common cdf $F$, then 
\begin{equation}
\label{eq:MDA}
    \dfrac{\max_{i=1,\ldots,n}\bxi_{i}- \boldsymbol{b}_{n}}{\boldsymbol{a}_{n}} \dto G, \qquad n \to \infty.
\end{equation}
The margins of an mev distribution are univariate extreme-value distributions themselves. Thus, for $j \in D$, the marginal distribution $G_j$ belongs to a three-parameter family and is
\begin{equation*}
G_j(x_j) =
\begin{dcases}
 \exp\sbr{-\cbr{1+\gamma_j\rbr{x_j-\mu_j}/\alpha_j }^{-1/\gamma_j}},& \text{if } \gamma_j \neq 0,\\
\exp\sbr{-\exp\cbr{-\rbr{x_j-\mu_j} / \alpha_j}}, & \text{if } \gamma_j=0,
\end{dcases}
\end{equation*}
for some shape parameter $\gamma_j \in \reals$, location parameter $\mu_j \in \reals$, and scale parameter $\alpha_j>0$; further, $x_j$ needs to be such that $1+\gamma_j(x_j-\mu_j)/\alpha_j >0$. 


\jsch{For $d\geq 2$, the family of $d$-dimensional mev distributions no longer constitutes a finite-dimensional parametric family, in contrast to the univariate case. This is due to the nature of the dependence structure between the components of $G$, or equivalently, the extremal dependence between the components of $\bxi$.} 
Many functions and measures that describe this dependence are discussed in the literature \citep[see, e.g.,][Chapter~8]{Be04}. \jsch{In this paper, three such objects will appear on the forefront: the exponent measure, the angular measure \citep{de1977limit} and the stable tail dependence function \citep{dh:1998}.}

\jsch{To concentrate on the dependence structure of $G$, it is convenient to standardize its margins. The choice of common scale is a matter of convention and convenience. In this section, we will opt for the unit-\FR{} distribution, i.e., $\alpha_j=\gamma_j=\mu_j=1$ for all $j \in D$, in line with the theory as originally exposed in \citep{de1977limit}. In Section~\ref{sec:backgroundmgpd}, however, we will choose the Gumbel distribution instead ($\alpha_j = 1$ and $\gamma_j=\mu_j=0$ for all $j \in D$), in order to link up with the additive model formulation for the mgp in \citet{R18}. Switching between different marginal scales is an easy and sometimes illuminating exercise.}

An mev distribution $G$ with general margins admits the representation
  \begin{equation*}
    G(\bx)=G^{*} \rbr{-1/\ln G_1(x_1), \ldots, -1/\ln G_d(x_d)},  
    \label{equa:mevdstandard} 
  \end{equation*}  
where $G^{*}$ is an mev distribution with unit-\FR~margins. For $\bxi_i$ as in \eqref{eq:MDA}, if $\bX_i=\rbr{X_{i1},\ldots,X_{id}}$ is defined by
\amch{
\begin{equation}
    X_{ij}=-1/\ln F_j\rbr{\xi_{ij}},
    \qquad i \in \mathbb{N}, \; j \in D,
    \label{equa:standardization}
\end{equation}
}
then we have
$\max_{i=1,\ldots,n}\bX_{i}/n \dto G^{*}$ as $n \to \infty$.



The \emph{exponent measure} of $G^{*}$ is the unique Borel measure $\mu$ concentrated on $\EE$ such that 
\begin{equation}
    G^{*}(\bx)= \exp\cbr{-\mu\rbr{\EE \setminus \left[\bzero,\bx\right]}}, 
    \qquad  \bx \in [\bzero, \binfty].
    \label{equa:exponentmeasure}
\end{equation}
The measure $\mu$ is finite on Borel sets whose topological closure does not contain the origin, i.e., is bounded away from the origin. Moreover, $\mu\rbr{\EE \setminus \left[\bzero,\bx\right]} = \infty$ as soon as $x_j=0$ for some $j \in D$. 
The exponent measure $\mu$ \amch{is $(-1)$-homogeneous} in the sense that $\mu(c \point)=c^{-1} \mu(\point)$ for all $c>0$. Moreover, for any $\mu$-continuous Borel set $B \subset \EE$ bounded away from the origin, we have 
\begin{equation}
    \lim_{n \to \infty} n \pr(\bX \in n B) =  \mu(B). \label{09:12:23:49}
\end{equation}

Another way to describe the dependence structure of $G$ is via the \emph{stable tail dependence function (stdf)},
\begin{equation}
     \ell(\bx)
     = \mu\rbr{\EE \setminus [\bzero,\bone/\bx]}
     = - \ln G^{*}(\bone/\bx), \qquad \bx \in [\bzero,\binfty).
\label{equa:defstdf}
\end{equation}
The margins are constrained by $\ell(0,\ldots,0,x_j,0,\ldots,0)=x_j$ for $x_j \ge 0$ and $j \in D$. Moreover, $\ell$ is homogeneous in the sense that for $c \ge 0$ and $\bx \in [\bzero,\binfty)$, we have $\ell(c\bx)=c\ell(\bx)$.

The \emph{angular measure} $\phi$ of $G^{*}$ with respect to a general norm $\|\,\cdot\,\|$ on $\Rd$ is 
\begin{equation} 
    \phi(B) =\mu \rbr{ \cbr{\bx \in \EE: \; \|\bx\|>1, \; \bx /\|\bx\| \in B } },  
    \qquad B \in \borel(\mathbb{S}).
    \label{09:12:23:48}
\end{equation}
Equations~\eqref{09:12:23:49} and \eqref{09:12:23:48} yield
\begin{equation*}
    \lim_{n \to \infty} n\pr(\|\bX\|>n, \; \bX/\|\bX\| \in B)= \phi(B),
    \label{09:26:14:56} 
\end{equation*}
for any $\phi$-continuous Borel set $B \subseteq S$.

\subsection{Multivariate generalized Pareto distributions}
\label{sec:backgroundmgpd}


The peaks-over-thresholds method is grounded in the fact that asymptotically, vectors of component-wise maxima follow an mev distribution if and only if exceedances over a high threshold vector follow an mgp distribution. Recall $\bxi$ and $G$ in \eqref{eq:MDA} and suppose $G(\bzero)>0$. Convergence~\eqref{eq:MDA} is satisfied if and only if there exists a distribution $H$ such that 
\begin{align*}
     \law\left(\dfrac{\bxi-\boldsymbol{b}_{n}}{\boldsymbol{a}_{n}} \vee \bo{\zeta} \; \middle| \; \bxi \nleqslant \boldsymbol{b}_{n}\right)  \dto H,
     \qquad n \to \infty,
\end{align*}
where $\bo{\zeta}$ is the vector of lower endpoints of $G_1,\ldots,G_d$ and $\vee$ denotes the component-wise maximum. 
\jsch{The conditioning event $\bxi \nleqslant \bo{b}_n$ is to be read as $\exists j \in D : \xi_j > b_{nj}$, that is, at least one variable within $\bxi$ exceeds a high threshold.}
We write $H=\GP(G)$. As in \citet{R18}, we consider the standardized case where $G$ has Gumbel margins, i.e., $\gamma_j=\mu_j=0$ and $\alpha_j=1$ for all $j \in D$. The distribution of an mgp random vector $\bY \sim H$ is then determined by the stdf $\ell$. 

Recall the exponent measure $\mu$ of $G^{*}$. By \citet[Proposition 3.1]{TR06}, the law of $\ee^{\bo{Y}}$ can be expressed in terms of $\mu$ on the collection of sets $\mathbb{A}= \cbr{[\bo{0},\by], \; \by \in \left[\bzero,\binfty\right)}$, via 
\[
  \pr[\ee^{\bo{Y}} \leq \by]
    =\dfrac{\mu\rbr{\left[\bo{0},\by\right] \cap \cbr{\bx \geq \bzero: \; \bx \nleqslant \bone }}}{\mu\rbr{\cbr{\bx \geq \bzero: \; \bx \nleqslant \bone } } }.
\]
Since a probability measure on $\left[\bzero,\binfty\right)$ is determined by its values on $\mathbb{A}$, we get 
\begin{equation}
    \pr\sbr{\ee^{\bo{Y}} \in B}
    =\frac{\mu\rbr{B \cap \cbr{\bx \geq \bzero: \; \bx \nleqslant \bone}}}{\mu\rbr{ \cbr{\bx \geq \bzero : \; \bx \nleqslant \bone} } },
    \qquad B \in \borel\rbr{\left[\bzero,\binfty\right)}.
    \label{11:21:01:39}
\end{equation}
In words, the law of $\ee^{\bY}$ is equal to the restriction of $\mu$ to $\left\{\bx \geq \bzero : \; \bx \nleqslant \bone \right\}$, normalized to a probability measure.
Conversely, since $\mu \rbr{ \cbr{\bx \ge \bzero : \; x_j > 1 } } = - \ln G^*_j(1) = 1$ by \eqref{equa:defstdf}, we obtain the normalization constant via $\pr \sbr{Y_j > 0} = 1 / \mu\rbr{ \cbr{\bx \geq \bzero : \; \bx \nleqslant \bone} }$ for all $j \in D$, so that $\mu$ can be reconstructed from $\bY$ by
\begin{equation}
\label{eq:Y2mu}
	\mu\rbr{B \cap \cbr{\bx \geq \bzero: \; \bx \nleqslant \bone}}
	= \pr\sbr{\ee^{\bo{Y}} \in B} / \pr\sbr{Y_j > 0},
	\qquad B \in \borel\rbr{\left[\bzero,\binfty\right)}.
\end{equation}


\citet[Section~4]{R18} propose three stochastic representations for $H$. In view of their importance to the construction device and the simulation algorithm to follow, we review these representations in some detail.

\paragraph{Spectral random vector $\boldsymbol{S}$.}
Let $\boldsymbol{S}$ be a random vector in $[-\infty,0]^d$ such that $\pr[\max(\bS)=0]=1$, $\pr[S_j>-\infty]>0$ for all $j \in D$ and $\expec[\ee^{S_1}]=\ldots=\expec[\ee^{S_d}]$. Then $\bS$ is said to be the \emph{spectral random vector} \amch{\citep[Theorem~6]{R18}} associated to the mgp $H$ if 
\begin{equation}
\ell(\by)=\expec[\max\rbr{\by \ee^{ \boldsymbol{S}} }] /\expec[\ee^{S_1}], \qquad \by \in \left[\bzero,\binfty\right).
    \label{equa:representationS}
\end{equation}
We write $H=\GP_{\bo{S}}(\law(\boldsymbol{S}))$, while $Y \sim H$ admits the representation
    \begin{align}
        \bY =  E+\boldsymbol{S},
        \label{ZtoS}
    \end{align} 
with $E = \max(\bY)$ a unit exponential random variable independent of the spectral random vector $\bS = \bY - \max(\bY)$\jsch{; see \citet[Theorem~7]{R18}}.

\jsch{The distribution of the random vector $\bS$ can thus be obtained directly from the mgp $H$.} Conversely, any random vector $\bS$ with the above properties generates an mgp $H$. \jsch{Still, the constraint $\max(\bS) = 0$ almost surely is not satisfied by common families of distributions, which raises the question how to construct appropriate $\bS$. The next representation resolves this issue, thereby facilitating the construction of an mgp.}

    \paragraph{Generator $\bT$.} Let $\bT$ be a random vector in $\left[-\infty,\infty\right)^d$ such that $\pr[\max(\bT)>-\infty]=1$, $\pr[T_j>- \infty] >0$ for all $j \in D$ and  $\expec[\ee^{T_1 -\max(\bT)}]=\ldots=\expec[\ee^{T_d-\max(\bT)}]$. The random vector $\bo{T}$ is called a \emph{$\bT$-generator} of $H$ \amch{\citep[Proposition~8]{R18}}, if we have the equality in distribution
     \begin{equation}
        \bo{S} \eqd \bT-\max(\bT), 
        \label{eq:STT}
     \end{equation}
where $\bS$ is the spectral random vector associated to $H$. We will write $H=\GP_{\bT}(\law(\bT))$. Equations~\eqref{ZtoS} and~\eqref{eq:STT} yield that if $\bT$ is a $\bT$-generator of $H$ and if $\bY \sim H$, then 
    \begin{equation}
        \bY \eqd  \bT - \max(\bT) +E.
        \label{ZtoT}
    \end{equation}

\jsch{Any random vector $\bT$ with the above properties yields an mgp random vector $\bY$ via \eqref{ZtoT}. In comparison to $\bS$, it has one constraint less to satisfy; the construction in \eqref{eq:STT} guarantees that $\max(\bS)$ vanishes almost surely. On the other hand, the $\bT$-generator of an mgp distribution $H$ is not unique: adding the same random variable to all components of $\bT$ does not change the distribution of $\bT - \max(\bT)$ and produces the same mgp $H$.}

\jsch{A drawback of both representations $\bS$ and $\bT$ is that they are not stable with respect to marginalization. If $\bY \sim H$ is a $d$-dimensional mgp random vector and if $J \subset D$ is non-empty, the conditional distribution of $\bY_J$ given $\max \bY_J > 0$ is a $|J|$-dimensional mgp, say $H_{|J}$ (which is different from $H_J$, the unconditional distribution of $\bY_J$, which is in general not a mgp). However, if $\bS$ is the spectral random vector of $H$, then $\bS_J$ is \emph{not} the spectral random vector of $H_{|J}$. The following representation resolves this issue, at the price of a change of measure.}

\paragraph{Generator $\bU$.} Let $\bU$ be a random vector on $[-\infty, \infty)^d$ such that $\expec[\ee^{U_1}]=\ldots=\expec[\ee^{U_d}]$ and $0<\expec[\ee^{U_1}]<\infty$. The random vector $\bU$ is said to be a \emph{$\bU$-generator} of $H$ \amch{\citep[Proposition~9]{R18}} if
\begin{equation}
	\ell(\by)
	=
	\expec \sbr{\max\rbr{\by \ee^{\bU} }} /\expec\sbr{\ee^{U_1}}, \qquad 
	\by \in \left[\bzero,\binfty\right).   
\label{eq:generator}
\end{equation}
 We write $H=\GP_{\bU}( \law(\bU))$. 
 If $\bU$ is a generator of $H$ and if $c \in \reals$, then $\bU+c$ is also a generator of $H$. This means that without loss of generality, we can and will assume that 
 \begin{equation}
        \expec\sbr{\ee^{U_j}} = 1   , \qquad j \in D.   \label{equa:conditiononU}         
 \end{equation}
\jsch{This constraint still does not make $\bU$ identifiable from $H$: if the random variable $V$ satisfies $\expec\sbr{\ee^V} = 1$ and is independent of $\bU$, then the random vector $(U_1+V,\ldots,U_d+V)$ yields another $\bU$-generator for the same mgp $H$.}

\jsch{If $\bU$ is a generator of $\bY \sim H$, then, for non-empty $J \subset D$, the random vector $\bU_J$ is a generator of $\bY_J \mid \max(\bY_J) > 0$, which thus follows a $|J|$-variate mgp \citep[Proposition~18]{R18}. In Proposition~\ref{propo:extremalfunction} below, we establish a link between $\bU$ and the extremal functions in \citet{dombry2016exact}. Equation~\eqref{eq:generator} represents the stdf $\ell$ as a $D$-norm with generator $\bU$ in the sense of \citet{MF19}. In \citet[Section~7]{A16}, common parametric families of multivariate extreme value models are described in terms of their $\bU$-generator.}

The spectral random vector $\bS$ can be obtained from the generator $\bU$ via  
    \begin{equation}
        \pr[\boldsymbol{S} \in B]= \frac{\expec\sbr{\indicator\cbr{\bU-\max(\bU) \in B} \ee^{\max(\bU)}}}{\expec\sbr{\ee^{\max(\bU)}}},
        \qquad B \in \borel\rbr{[-\binfty, \bzero)}.
        \label{09:29:14:29}
    \end{equation}
Equations~\eqref{eq:STT} and~\eqref{09:29:14:29} link representation $\bo{S}$ to $\bo{T}$ and to $\bo{U}$, respectively. Likewise, a $\bT$-generator can be obtained from the $\bU$-generator via 
\begin{equation}
    \pr[\boldsymbol{T} \in B]= \dfrac{\expec\sbr{\indicator\cbr{\bU \in B} \ee^{\max(\bU)}}}{\expec\sbr{\ee^{\max(\bU)}}},
    \qquad B \in \borel\rbr{[-\binfty,\binfty)}.
    \label{equa:TtoU}
\end{equation}
\jsch{Note that a spectral random vector $\bS$ is also a $\bT$-generator, since $\max(\bS) = 0$ and thus $\GP_{\bT}(\law(\bS)) = \GP_{\bS}(\law(\bS))$. On the other hand, the representation via $\bU$ is of another nature: if $\bV$ is a random vector whose distribution satisfies both the constraints to be a $\bT$-generator and a $\bU$-generator, then the mgp distributions $\GP_{\bT}(\law(\bV))$ and $\GP_{\bU}(\law(\bV))$ are different in general; see \citet[Section~7]{A16} for examples.}


Let $\phi$ be the angular measure of $G^{*}$ with respect to a norm $\norm$ on $\Rd$ \amtwo{as in Equation~\eqref{09:12:23:48}. Let $\bV$ be a random vector on the $\norm$-unit sphere $\sphere$ with distribution $\phi/\phi(\sphere)$.}
The random vector $\bU^{\phi}=\ln(\bV)$ is then a $\bU$-generator of $H$. For non-empty $J \subseteq D$, we have
    \begin{equation*}
    \label{equa:Utoangularmeasure} 
        \pr\sbr{U^{\phi}_j>-\infty \; \text{iff} \; j \in J}
        = \phi\rbr{\cbr{\bw \in \sphere: \; w_j>0 \; \text{iff} \; j \in J}} / \phi(\sphere),
    \end{equation*}
an equality that makes the link to extreme directions of the underlying random vector $\bxi$, as treated next. 

\section{Extreme directions}
\label{sec:Directions}



\amch{Let the random vector $\bX = (X_j)_{j \in D}$ be defined on the basis of the probability integral transform applied to the components of a random vector $\bxi$ as in \eqref{equa:standardization},} and assume the multivariate regular variation condition $n \pr\sbr{\bX \in n \point} \to \mu(\point)$ in the sense of \eqref{09:12:23:49} holds. We are interested in scenarios where all variables $X_j$ for $j$ in some subset $J$ of $D$ are large while the other variables are not; a precise meaning will be given below. Such a group of variables $J$ will be called an extreme direction, defined formally in Section~\ref{subsec:Sparde mevd}. We will provide various characterizations of such extreme directions: by the angular measure, by the asymptotic probability of $\bX$ to belong to certain sets called thickened rectangles and cones, and finally by the mgp distribution. Next, in Section~\ref{sec:sparse mgpd}, we express the conditional probability density function of an mgp random vector given a certain extreme direction in terms of other densities that are easier to compute.

\amch{Upper tail dependence coefficients \citep{coles1999dependence, schlather2002inequalities}
\begin{equation}
\label{eq:chiJ}
	\chi_{J} 
	= \lim_{q \to \infty} q \pr\sbr{\forall j \in J : X_j > q},
	\qquad \emptyset \ne J \subseteq D,
\end{equation}
are well-known summary measures of extremal dependence. However, we would like to emphasize that they are not sufficient to characterize extreme directions. In fact, even though the equality $\chi_{J} = 0$ implies that the variables $X_j$ for $j \in J$ cannot be large simultaneously, the inequality $\chi_{J} > 0$ does not exclude the possibility of extreme directions properly contained in $J$. For instance, for a pair $J = \cbr{j_1, j_2}$, the inequality $\chi_{j_1j_2} > 0$ implies that $X_{j_1}$ and $X_{j_2}$ can be large simultaneously but does not rule out the possibility that $X_{j_1}$ is large while $X_{j_2}$ is not.}

\bgroup
\color{gray}
\egroup

\subsection{Definition and characterizations}
\label{subsec:Sparde mevd}
Recall the exponent measure $\mu$ of $G^{*}$ as defined in \eqref{equa:exponentmeasure} and $\EE=[\bzero,\binfty) \setminus \cbr{\bzero}$. In this section, let $J \subseteq D$ be a fixed non-empty set and $\EE_J:=\cbr{\bx \in \EE: \ x_j>0 \text{ iff } j \in J}$. The collection $\cbr{ \EE_J : \; \emptyset \ne J \subseteq D}$ forms a partition of $\EE$.
The following definition was introduced in \citet{simpson2020determining} in order to identify extreme directions of the random vector $\bX$ defined as in \eqref{equa:standardization}.

\begin{definition}
\label{def:Xdir}
The set $J$ is said to be an extreme direction of $\bX$ or $\mu$ if $\mu(\EEJ)>0$.
\end{definition}

A trivial but useful fact is that $\mu$ has no mass outside $\bigcup_{J \in \mathcal{R}(\mu)} \EE_J$, where $\mathcal{R}(\mu)$ is the collection of extreme directions of $\mu$.
Further, even if a certain set $J \subseteq D$ is not an extreme direction of $\bX$, it is still possible that $J$ is a superset or a subset of an extreme direction. In contrast, if $\chi_J = 0$, then also $\chi_{\bar{J}} = 0$ whenever $\bar{J} \subseteq J$. Remark~\ref{remark:chi} below points out a connection between tail dependence coefficients and extreme directions.

Throughout this section, let $\norm$ be a fixed arbitrary norm on $\Rd$ and let $\phi$ denote the angular measure of $G^{*}$ with respect to this norm. Recall the set $\mathbb{S}=\cbr{\bx \in \EE: \; \|\bx\|=1}$ and let $\mathbb{S}_J:=\cbr{\bw \in \mathbb{S}: \; w_j>0 \text{ iff } j \in J}$. The following proposition gives a characterization of extreme directions in terms of the angular measure $\phi$. 

\begin{proposition}
\label{propo:angularmeasureinterpretation}
   The set $J$ is an \extreme~of $\bX$ if and only if $\phi(\mathbb{S}_J)>0$.
\end{proposition}

The special case of Proposition~\ref{propo:angularmeasureinterpretation} for the $L_\infty$-norm will be used below for the characterization of extreme directions in terms of the mgp distribution.


Let $\eps>0$ and $n \in \mathbb{N}$. We will interpret extreme directions in terms of the asymptotic probability of $\bX$ to be contained in certain thickened rectangles and cones associated with $\norm$, defined as
\begin{align}
\rectangle &= \left\{
        \bx \in \EE : \
        \| \bx \|> n, \; 
      \bx_J/n \in (\eps,\infty)^{J},\;
      \bx_{D \setminus J}/n \in [0,\eps]^{\lvert D \setminus J \rvert}
    \right\}, \label{thickrectangle} \\
\halfcone  &= \left\{
		\bx \in \EE: \;  
		\| \bx\|>n, \; 
		\bx_J / \| \bx\| \in (\eps,\infty)^{J},\; \bx_{D \setminus J} / \| \bx\| \in [0,\eps]^{\lvert D \setminus J \rvert}  
		\right\}, \label{thickcone}
\end{align}
respectively, and illustrated in Figure~\ref{figurethick}. 

\begin{figure}
  \begin{center}
\begin{tikzpicture}
\draw (2,0) node [below]{$n$};
\draw (0,2) node [left]{$n$};
\draw(0,0) -- (2,0);
\draw (5,0.05) node[right] {$\mathbb{R}_{n,\cbr{1}}^{\eps}$};
\draw (0.2,5) node[above] {$\mathbb{R}_{n,\cbr{2}}^{\eps}$};
\draw (0,0) -- (0,2);

\draw [dashed] (0.25,1.75) -- (0.25,0) node[below] {$n\eps$};
\draw [dashed] (1.75,0.25) -- (0,0.25) node[left] {$n\eps$};
\draw  (0.25,1.75) -- (1.75,0.25); 
\draw  (0.25,2) -- (2,2)-- (2,0/25); 
\filldraw[fill=gray!20,gray!20]
(0,2) -- (0,5) -- (0.25,5) -- (0.25,1.75) -- cycle;
\draw [dashed, red] (0,2) -- (0,5) -- (0.25,5) -- (0.25,2)-- (0,2);
\filldraw[fill=gray!20,gray!20]
(2,0) -- (5,0) -- (5,0.25) -- (1.75,0.25) -- cycle;
\draw [dashed, red] (2,0) -- (5,0) -- (5,0.25) -- (2,0.25)-- (2,0);
\draw (1,1) node [above,rotate=-45]{$\|\bx\|_1=n$};
\draw (1.11,1.9) node [above]{$\|\bx\|_{\infty}=n$};
\end{tikzpicture}
\begin{tikzpicture}
\draw (2,0) node [below]{$n$};
\draw (0,2) node [left]{$n$};

\draw (5,0.05) node[right] {$\mathbb{C}_{n,\cbr{1}}^{\eps}$};
\draw (0.2,5) node[above] {$\mathbb{C}_{n,\cbr{2}}^{\eps}$};
\draw (0,0) -- (0,2);
\draw (0,0) -- (2,0);

\filldraw[fill=gray!20,gray!20]
(0,2) -- (0.25,5/3) -- (3/4,5) -- (0,5) -- cycle;
\filldraw[fill=gray!20,gray!20]
(2,0) -- (5/3,0.25) -- (5,3/4) -- (5,0) -- cycle;
\draw [dashed, red] (0,2) -- (0.25,2) -- (5/8,5) -- (0,5)-- (0,2);
\draw [dashed, red] (2,0) -- (2,0.25) -- (5,5/8) -- (5,0)-- (2,0);
\draw  (5/3,0.25) -- (0.25,5/3); 
\draw  (0.25,2) -- (2,2)-- (2,0.25);
\draw (1,0.8) node [above,rotate=-45]{$\|\bx\|_1=n$};
\draw (1.15,1.9) node [above]{$\|\bx\|_{\infty}=n$};
\draw (3.5, 0.5) node [above, rotate=10, ] {\tiny $x_2=\eps(1-\eps)^{-1}x_1$};
\draw (3.5, 0.3) node [red, rotate=10] {\tiny $x_2=\eps x_1$};
\end{tikzpicture}
\end{center}
    \caption{Thickened rectangles as in \eqref{thickrectangle} (left) and thickened cones as in \eqref{thickcone} (right) associated with the $L_{1}$-norm (simple gray region) and $L_{\infty}$-norm (gray region with dashed red line boundary).}
    \label{figurethick}
\end{figure}

The latter sets were introduced in \citet{goix2017JMVA, Go17} for the $L_{\infty}$-norm. For $\eps>0$, the Borel set $\mathbb{R}_{1,J}^{\eps}$ is $\mu$-continuous as defined in Section~\ref{sec:background}. Hence we can make the following link with Proposition~\ref{propo:angularmeasureinterpretation}.

\begin{lemma} 
\label{lemma:halfrectangles}
We have 
\begin{align}
    \phi(\sphereJ)
    &=\lim_{\eps \downarrow 0 } \lim_{n \to \infty} n  \pr\left[\bX \in \rectangle \right].
    \label{equation:halfrectangles}
\end{align} 
\end{lemma}
In contrast, a standardization by $\|\bx\|$ rather than $n$ is used to define thickened cones. Thus, for $\eps>0$, the Borel set $\mathbb{C}_{1,J}^{\eps}$ is not necessarily $\mu$-continuous. To obtain a similar limit to \eqref{equation:halfrectangles} for the thickened cones, we should exclude at most countably many values of $\eps > 0$ for which $\mathbb{C}_{1,J}^{\eps}$ is not $\mu$-continuous.

\begin{lemma}
\label{lemma:halfcones}
There exists an at most countable set $N \subset [0, \infty)$ such that 
\begin{align}
\label{equality:halfcones}
    \phi(\sphereJ)
    &=\lim_{\substack{\epsilon \downarrow 0 \\ \epsilon \not\in N} } \lim_{n \to \infty} n  \pr\left[\bX \in \mathbb{C}_{n,J}^{\epsilon}\right].
\end{align}     
\end{lemma}

Since the sets $\mathbb{C}_{n,J}^{\epsilon}$ depend in a monotone way on $\epsilon$, the formula~\eqref{equality:halfcones} also holds without excluding any values of $\epsilon$, provided the limit over $n$ is replaced by a limsup. 

Proposition~\ref{propo:angularmeasureinterpretation} and Lemmas~\ref{lemma:halfrectangles} and~\ref{lemma:halfcones} imply that the set $J$ is an extreme direction of $\bX$ if and only if the equivalent conditions 3.\ and 4.\ in Theorem~\ref{Alllinked} are satisfied.

Finally, we characterize extreme directions in terms of the mgp distribution and its representations. Let $\tilde{G} 
(\bx)=G^{*}\rbr{\exp\rbr{x_1},\ldots,\exp\rbr{x_d}}$ for $\bx=(x_1,\ldots,x_d)$, be the Gumbel standardization of $G^{*}$ and consider the associated mgp distribution $H=\GP(\tilde{G})$ as in Section~\ref{sec:backgroundmgpd}. Recall representations $\bS$, $\bT$ and $\bU$ of $H$ in Eqs.~\eqref{equa:representationS}--\eqref{eq:generator}. Let $\sphere_{\infty} = \cbr{\bx \ge \bzero : \ \| \bx \|_\infty = 1}$ denote the positive $L_{\infty}$-sphere and $ \sphere_{J,\infty}:=\cbr{\bw \in \sphere_{\infty}: \ w_j>0 \text{ iff } j \in J } $ for non-empty $J \subseteq D$. Let $\bY \sim H$. Applying Eq.~\eqref{11:21:01:39} to the set $\EE_J$ gives
\[
	\pr\sbr{Y_j>-\infty \text{ iff } j \in J}
	= \frac{\phi_{\infty} \rbr{ \sphere_{J,\infty} } }{ \phi_{\infty} (\sphere_{\infty})},
\]
where $\phi_{\infty}$ is the angular measure of $G^{*}$ with respect to the $L_{\infty}$-norm. Proposition~\ref{propo:angularmeasureinterpretation} applied to the $L_\infty$-norm yields that $J$ is an extreme direction  of $\bX$ if and only if $\pr \sbr{Y_j>-\infty \text{ iff } j \in J} > 0$. The following lemma expresses this probability in terms of representations $\bS$, $\bT$ and $\bU$.

\begin{lemma}
\label{linkbetweenrepresntations}
For non-empty $J \subseteq D$, we have 
\begin{align}
	\pr \sbr{Y_j> -\infty \text{ iff } j \in J}
	&=\pr\sbr{S_j> -\infty \text{ iff } j \in J} \label{YtoS}\\
	&=\pr\sbr{T_j> -\infty \text{ iff } j \in J} \label{YtoT}\\
	&=\expec\sbr{\ee^{\max(\bU_J)} \indicator \cbr{U_j> -\infty \text{ iff } j \in J}}/\expec[\ee^{\max\bU} ]. \label{YtoU}  
\end{align}
\end{lemma}

We finally get the following theorem which summarizes this section.

\begin{theorem} 
\label{Alllinked}
For non-empty $J \subseteq D$, the following statements are equivalent:
\begin{enumerate}
    \item  The set $J$ is an \extreme.
    \inlineitem $\phi(\sphereJ)>0$.  \label{item:angularmeasure} 
  \item  $\displaystyle\lim_{\eps \downarrow 0 } \lim_{n \to \infty} n  \pr\sbr{\bX \in \mathbb{R}_{n,J}^{\epsilon}}>0$. \label{item:rectangles}
    \inlineitem $\displaystyle\lim_{\eps \downarrow 0} \limsup_{n \to \infty} n  \pr\sbr{\bX \in \mathbb{C}_{n,J}^{\epsilon}} > 0$. \label{item:halfcones}
    \item $\pr[Y_j>-\infty \text{ iff } j \in J]>0$.
    \inlineitem $\pr[U_j>-\infty \text{ iff } j \in J]>0$. \label{item:U}
    \item $\pr[T_j>-\infty \text{ iff } j \in J]>0$.
    \inlineitem $\pr[S_j>-\infty \text{ iff } j \in J]>0$.
\end{enumerate}
\end{theorem}

\begin{example}[Max-linear model]
\label{ex:maxlin}
Let $r$ be a positive integer and write $R = \cbr{1,\ldots,r}$. Let $Z_1,\ldots,Z_r$ be independent unit-\FR{} random variables. Let $A = (a_{jk})_{j \in D;k \in R} \in [0,1]^{d \times r}$ be a coefficient matrix with row sums $\sum_{k=1}^{r} a_{jk}=1$ for all $j \in D$ and column sums $\sum_{j=1}^d a_{jk} > 0$ for all $k \in R$. Consider the random vector
\begin{equation}
    \label{equa:max-line}
    \bM=\left(\max_{\ka \in \bigar}\{a_{1\ka} Z_{k}\}, \ldots, \max_{\ka \in \bigar}\{a_{d\ka} Z_{k}\}\right). 
\end{equation}  
Its distribution is a max-linear model; see, e.g., \citet{wang2011conditional} and \citet{Ei12}. 
\jsch{As in linear factor models, the random variables $Z_1,\ldots,Z_r$ can be thought of as latent factor variables generating the observable random vector.}
The random vector $\bM$ has an mev distribution $G^{*}$ with unit-\FR{} margins. The angular measure of $G^{*}$ with respect to $\norm$ is a discrete measure with mass $h_k=\|\bo{a}_{\point k}\|$ at atom $\bo{w}_{k}= \bo{a}_{\point k} / h_k$, where $\bo{a}_{\point k}$ denotes the $k$-th column of $A$ for $k \in R$. By Proposition~\ref{propo:angularmeasureinterpretation}, the extreme directions of $\bM$ are the sets $J_k=\cbr{j \in D: \; a_{jk}>0 }$ for $k \in R$. \amtwo{The construction at the end of Section~\ref{sec:backgroundmgpd} yields a generator $\bU$ of the mgp distribution associated to \eqref{equa:max-line} with distribution
\begin{equation*}
    \law(\bU) = \frac{1}{d} \sum_{k \in R} \|\bo{a}_{\cdot k}\|_1 \, \delta_{\ln \rbr{\tfrac{\bo{a}_{\cdot k}}{\|\bo{a}_{\cdot k}\|_1}} } \rbr{\cdot},
\end{equation*}
with $\delta_{\bx}(\cdot)$ denoting a unit point mass at $\bx$. Hence, inequality~(6) in Theorem~\ref{Alllinked} is satisfied if and only if $J$ is one of the signatures $J_k, \, k\in R$.}
\end{example}



\begin{remark}
\label{remark:chi}
The tail dependence coefficient $\chi_J$ in \eqref{eq:chiJ} satisfies
\[ 
    \chi_J = \mu \left( \left\{ \bx \in \EE : \; \bx_J \in (1,\infty)^{J}\right\} \right)
    = \sum_{\substack{\bar{J} \subseteq D\\ \bar{J}  \supseteq J }} \mu\left( \left\{ \bx \in \EE: \; \bx_{\bar{J}} \in (1,\infty)^{\lvert \bar{J} \rvert}, \; \bx_{D\setminus \bar{J}} \in [0,1]^{\lvert D \setminus \bar{J} \rvert }  \right\} \right).
\]
We have $\chi_J > 0$ if and only if there exists $\bar{J} \supseteq J$ such that the corresponding term on the right-hand side is positive, and by homogeneity of $\mu$, this is equivalent to $\bar{J}$ being a subset of an extreme direction. We find that $\chi_J > 0$ if and only if $J$ is a subset of an extreme direction. 
\end{remark}

\subsection{Density of an mgp distribution with multiple extreme directions}
\label{sec:sparse mgpd}

We provide density expressions for potential use in statistical inference on mgp models.
As in Section~\ref{sec:backgroundmgpd}, let $\bY$ be an mgp random vector with distribution $H = \GP(G)$ where $G$ is an mev distribution with Gumbel margins. Further, let $\bT$ and $\bU$ be its generators as defined in Section~\ref{sec:background}. With respect to a dominating measure $\upsilon$, defined as a sum of Lebesgue measures $\lambda_m$ of various dimensions $m \in D$, we express the density of $\bY$ in terms of the densities of $\bU$ and $\bT$. This will be helpful since the latter densities are known for several parametric families, for instance, the logistic and the Hüsler--Reiss model \citep[Sections~7.1 and~7.2]{A16}. Furthermore, in view of \eqref{eq:Y2mu}, the density of the exponent measure $\mu$ on $\cbr{\bx \ge \bzero : \bx \nleqslant \bone}$ is proportional to the mgp density, after appropriate marginal scaling; by homogeneity, this determines the exponent measure density everywhere.

Let $\lambda_k$ be the $k$-dimensional Lebesgue measure and let $\upsilon = \upsilon_d$ be the measure on $[-\infty,\infty)^d$ defined by 
\begin{equation}
\label{eq:upsilon}
     \upsilon(\point) := \sum_{J : \emptyset \ne J \subseteq D} 
     \lambda_{\lvert J \rvert} (\pi_{J}(\point \cap \mathbb{A}_{J})),
\end{equation}
where we recall the set $\mathbb{A}_J$ and the projection $\pi_J$ defined in and right after Eq.~\eqref{eq:AJ}. For Borel sets $B$ in $[-\infty,\infty)^d$, we have thus
\[
    \upsilon(B) = \sum_{J : \emptyset \ne J \subseteq D} \lambda_{|J|} \rbr{\cbr{ \bx \in \reals^{|J|} : \; \text{$\exists \, \by \in B$ such that $y_j = x_j$ if $j \in J$ and $y_j = -\infty$ if $j \not\in J$}}}.
\]
\begin{theorem}
\label{theorem:pdfofUwrtmu}
If the generator $\bU$ of $\bY$ has density $f_{\bU}$ with respect to $\upsilon$, then the mgp random vector $\bY$ has density $h_{\bY}$ with respect to $\upsilon$ equal to
\begin{equation*}
    h_{\bY}(\by)=\indicator\cbr{\max(\by) >0} \frac{1}{\expec \sbr{\ee^{\max(\bU)}}} \int_{r \in \reals } f_{\bU}\rbr{r+\by} e^r \diff r.
\end{equation*}
\end{theorem}

\begin{theorem}
\label{theorem:pdfofTwrtmu}
If the generator $\bT$ has density $f_{\bT}$ with respect to $\upsilon$, then the mgp random vector $\bY$ has density $h_{\bY}$ with respect to $\upsilon$ equal to
\begin{equation*}
     h_{\bY}(\boldsymbol{y})
    =  \indicator\cbr{\max(\by)>0} \ee^{-\max(\by)} \int_{r \in \reals} f_{\bT}(r+\boldsymbol{y}) \diff r.
\end{equation*}
\end{theorem}

The proofs of Theorems~\ref{theorem:pdfofUwrtmu} and~\ref{theorem:pdfofTwrtmu} are identical to the ones of Theorems~12 and~13 in \citet{R18}. In the next section, we will propose a model for the $\bU$-generator that leads to mgp distributions with general extreme directions and whose density can be computed by means of Theorem~\ref{theorem:pdfofUwrtmu}.

\section{A model with extreme directions}
\label{sec:sparsemgpds}

We will provide a stochastic construction that allows any collection $\mathcal{R}$ of non-empty subsets of $D = \cbr{1,\ldots,d}$ with the property $\bigcup \cbr{J : J \in \mathcal{R}} = D$ to be the extreme directions (Definition~\ref{def:Xdir}) of a random vector. We do so by considering a smoothed version of the max-linear model that we will call mixture model (Section~\ref{sec:definitionandproperties}). We show that this model determines an mev distribution, of which we compute the stdf, the extreme directions, and the $\bU$-generator of the associated mgp distribution. Importantly, the mixture model covers all mev distributions and thus also all mgp distributions. Density expressions will be provided too. Next, in Section~\ref{sec:examples}, we work out two parametric examples for the mixture model based on the logistic and the \HR~dependence structures.

\subsection{Definition and properties}
\label{sec:definitionandproperties}

\begin{definition}
\label{def:mixture}
Let $r$ be a positive integer and write $R=\{1,\ldots,r\}$. Let $\bo{Z}^{(1)},\ldots,\bo{Z}^{(r)}$ be independent random vectors in $\Rd$. 
For all $k \in R$, suppose that $\bo{Z}^{(k)} = (Z_1^{(k)}, \ldots, Z_d^{(k)})$ has an mev distribution with unit-Fr\'echet margins, stdf $\ell^{(k)}$ and a single extreme direction $D$. Let $A = (a_{jk})_{j \in D;k \in R} \in [0,1]^{d \times r}$ be a coefficient matrix with unit row sums $\sum_{k=1}^{r} a_{jk}=1$ for all $j \in D$ and positive column sums $\sum_{j=1}^d a_{jk} > 0$ for all $k \in R$. 
The $d$-variate random vector
\begin{equation}
  \label{mixturemodel}
    \bM 
    = (M_1, \ldots, M_d) 
    = \left(\max_{\ka \in \bigar}\{a_{1\ka} Z_{1}^{(\ka)}\}, \ldots, \max_{\ka \in \bigar}\{a_{d\ka} Z_{d}^{(\ka)}\}\right),   
\end{equation}
will be called the $(d\times r)$ \emph{mixture model} with coefficient matrix $A$ and factors 
$\bZ^{(1)},\ldots,\bZ^{(r)}$.  
\end{definition}

\begin{remark}
\label{rem:ZkJk}
In \eqref{mixturemodel}, for a given $j \in D$, only those $k \in R$ that satisfy $a_{jk} > 0$ matter for the value of the maximum. Therefore, only the subvectors $\bZ^{(k)}_{J_k}$ need to be given, where
\begin{equation}
\label{eq:Jk}
    J_k := \cbr{ j \in D: \; a_{jk} > 0 }, 
\end{equation}
is the $k$-th \emph{signature} of the coefficient matrix $A$. The condition on $\bZ^{(k)}_{J_k}$ is then that its only extreme direction is $J_k$. Note that $J_k$ is non-empty by the assumption that $A$ has positive column sums.
\end{remark}
 
\begin{remark}
\label{remark:smoothedmaxlinear}
   In case $Z_1^{(k)}=\ldots=Z_d^{(k)}$ almost surely for all $k \in R$, the distribution of $\bM$ is max-linear with unit-Fréchet margins \jsch{and we are back to the factor model in Example~\ref{ex:maxlin}}. 
   The mixture model~\eqref{mixturemodel} can thus be seen as a smoothed version of the max-linear factor model in the sense that complete dependence within each of the 
    factors $\bZ^{(1)},\ldots,\bZ^{(r)}$ is relaxed to a possibly weaker dependence structure. 
    \jsch{Under the conditions of Proposition~\ref{propo:densityofmixturemgp} below, the resulting mixture model has a density with respect to the measure $\upsilon$ in Eq.~\eqref{eq:upsilon}. We hope this property will facilitate statistical inference through likelihood-based methods.}
\end{remark}
For $k \in R$, let $\ell_{J_k}^{(k)}$ be the stdf associated with $\bZ^{(k)}_{J_k}$, that is,
\begin{equation}
    \label{equa:k-thstdf}
  \ell_{J_k}^{(k)}(\bx_{J_k})= \ell^{(k)}\rbr{\Tilde{\bx}}, 
  \qquad \bx_{J_k} \in [0, \infty)^{J_k},
\end{equation}
with $\Tilde{\bx}= \sum_{j \in J_k} x_j \bo{e}_j$ and $\bo{e}_j$ the $j$-th canonical unit vector in $\Rd$. If $J_k$ is a singleton, then $\ell_{J_k}^{(k)}$ is the identity function on $[0, \infty)$.
 The following theorem was proved in \citet[Theorem 1]{S03} for the particular case where $J_1,\ldots,J_r$ are 
all different. 

\begin{theorem}
\label{stephanson}
  The random vector $\bM$ in \eqref{mixturemodel} follows an mev distribution \amtwo{denoted by $G_{\bM}$} with unit-Fr\'echet margins and stdf 
\begin{equation}
\label{stdfgeneralmodel}
    \ell(\bx) = \sum_{k \in R} \ell^{(k)}_{J_k}\rbr{\rbr{a_{jk}x_j}_{j \in J_k}},
    \qquad \bx \in [0, \infty)^d.
\end{equation}
\end{theorem}

In the context of anomaly detection, \citet{chiapino2020} define a mixture model for the angular measure of $G_{\bM}$ in terms of probability measures on the faces of the unit simplex. In contrast, the mixture model in Definition~\ref{mixturemodel} is obtained by an aggregation of lower-dimensional mev distributions.

Zero entries of the coefficient matrix $A$ allow some groups of components of the random vector $\bM$ to constitute an \extreme. Recall $J_k$ in \eqref{eq:Jk}. 

\begin{proposition}
\label{propo:extremedirections}
The set of extreme directions of the random vector $\bM$ in \eqref{mixturemodel} is 
\[
    \mathcal{R} = \cbr{J \subseteq D : \; J=J_k \text{ for some } k \in R }.
\]
\end{proposition}

\begin{remark}
Given any collection $\mathcal{R} = \cbr{J_1,\ldots,J_r}$ of distinct non-empty subsets of $D$ with the property $\bigcup_{k=1}^r J_r = D$, we can easily construct a $d \times r$ coefficient matrix $A$ as in Definition~\ref{def:mixture} such that equality~\eqref{eq:Jk} holds for all $k$ (put zeros and ones on the appropriate places and then normalize the rows). The set of extreme directions of the mixture model~\eqref{mixturemodel} is then exactly $\mathcal{R}$.
\end{remark}

Let $\tilde{G}_{\bM}$ be the Gumbel standardization of $G_{\bM}$, that is, $\tilde{G}_{\bM}$ is the cdf of $(\ln M_j)_{j \in D}$. We will investigate the mgp distribution
\begin{equation}
H=\GP(\tilde{G}_{\bM})
\label{equa:H-mgp}
\end{equation}
as defined in Section~\ref{sec:backgroundmgpd}. We will refer to $H$ as the \emph{mgp distribution associated with the mixture model} $\bM$ and to $\bY \sim H$ as the \emph{mgp random vector associated with the mixture model} $\bM$. For a non-empty set $J \subseteq D$, recall the set $\mathbb{A}_J$,  the projection $\pi_J$ defined in Eq.~\eqref{eq:AJ}, \amtwo{and the representation $\bU$ of the mgp distribution $H$ as in Section~\ref{sec:backgroundmgpd}.} 
For $k \in R$, since the only extreme direction of \amtwo{the $\reals^{J_k}$-valued random vector $\bZ_{J_k}^{(k)}$} is $J_k$, Theorem~\ref{Alllinked} yields the existence of a \amtwo{generator} $\bU^{(k)}$ taking values in $\reals^{J_k}$ which satisfies \eqref{equa:conditiononU} and 
\begin{equation}
\label{equa:U^(k)}
    \expec \sbr{\max_{j \in J_k} \cbr{x_j\ee^{U_{j}^{(k)}}}}= \ell_{J_k}^{(k)} \rbr{\bx_{J_k}}, \qquad \bx_{J_k} \in [0,\infty)^{J_k}.
\end{equation}
 The following proposition provides a $\bU$-generator of $H$ in terms of $\bU^{(1)},\ldots,\bU^{(r)}$.

\begin{proposition}
\label{propo-generator-general-model}
   For any $\bo{m} \in (0,1)^{r}$ such that $\sum_{k=1}^r m_k=1$, the random vector $\bU$ in 
   $[-\infty,\infty)^d$ with distribution
       \begin{equation}
       \begin{aligned}
          \pr \sbr{\bU \in \point}
        = \sum_{k \in R} m_k \pr \sbr{ 
             \bU^{(k)}+\rbr{\ln\rbr{a_{jk}/ m_k}}_{j \in J_k} \in \pi_{J_k}\rbr{ \mathbb{A}_{J_k} \cap \point }
        }
         \label{equa:generalgenerator}
       \end{aligned}
    \end{equation}
    is a $\bU$-generator of the mgp distribution $H$ in \eqref{equa:H-mgp}, i.e., $H = \GP_{\bU}(\law(\bU))$.
\end{proposition}

The random vector $\bU$ in \eqref{equa:generalgenerator} will be called a mixture generator with coefficient matrix $A$, generators $\bU^{(1)},\ldots,\bU^{(r)}$ and mass vector $\bo{m}$. 
\jsch{Different choices of the vector $\bo{m}$ yield different $\bo{U}$ generators, all of which however induce the \emph{same} mgp distribution $H$, namely, the one associated to the mixture model. Among all possible choices for $\bo{m}$, the obvious one is $m_k = 1/r$ for all $k \in \cbr{1,\ldots,r}$, but the formula allows for other possibilities. Further, note that $m_k$ is \emph{not} equal to the probability that the mgp vector $\bY \sim H$ lies in extreme direction $J_k$: instead, we have $\pr[Y_j > -\infty \text{ iff } j \in J_k] = w_k$, a quantity which is proportional to $\ell_{J_k}^{(k)}((a_{jk})_{j\in J_k})$, as stated in Proposition~\ref{lemma:conditionalT's} below.}

Up to now, we introduced the mixture model and we investigated the associated mgp distribution. Conversely, the following theorem shows the generality of the construction in the sense that it accommodates all mgp distributions. We say that the coefficient matrix $A$ has \emph{mutually different signatures} if the sets $J_1,\ldots,J_r$ in \eqref{eq:Jk} are all distinct.

\begin{theorem}
\label{theorem:general}
Any mev distribution with unit-\FR~margins is the cdf of a mixture model whose coefficient matrix has mutually different signatures. As a consequence, any mgp distribution is associated with a mixture model whose coefficient matrix has mutually different signatures.
\end{theorem}

The coefficient matrix and the factors of the mixture model in Theorem~\ref{theorem:general} are 
specified in the proof. Applying Theorem~\ref{theorem:general} to the mixture model itself yields the following corollary.

\begin{corollary}
\label{corollary:anygenerator}
    Let $\bM$ be the mixture model in Definition~\ref{def:mixture}.
    There exist an integer $1 \leq s \leq r$ and a $\rbr{d \times s}$ mixture model $\bN$ whose coefficient matrix has mutually different signatures such that $\bM$ and $\bN$ have the same cdf and consequently the same associated mgp distribution. 
\end{corollary}

\begin{remark}
    On the one hand, Corollary~\ref{corollary:anygenerator} suggests to restrict mixture models to the case where 
    all columns of the coefficient matrix have different signatures. On the other hand, for statistical practice, it might be difficult to combine the dependence structures of two columns with the same signature and different stdf's.
\end{remark}

Let $\bY \sim H$ be the mgp random vector associated with the mixture model $\bM$. Using Theorem~\ref{theorem:pdfofUwrtmu}, we express the density of $\bY$ with respect to the measure $\upsilon$ in \eqref{eq:upsilon}.

\begin{proposition}
\label{propo:densityofmixturemgp}
Consider the mixture model $\bM$ in Definition~\ref{def:mixture} and recall its stdf $\ell$ in \eqref{stdfgeneralmodel}.
For all $k \in R$, suppose that the generator $\bU^{(k)}$ of $\bZ^{(k)}$ is absolutely continuous with respect to $\lambda_{\lvert J_k \rvert}$ with Lebesgue density $f_{\bU^{(k)}}$. The mgp random vector $\bY$ with distribution $H$ in \eqref{equa:H-mgp} is then absolutely continuous with respect to $\upsilon$ with density $h_{\bY}$ given for $\by \in [-\infty,\infty)^d$ with $\max(\by) > 0$ by 
    \begin{equation}
 h_{\bY}(\by)
 =   \sum_{k \in R} \frac{\indicator_{\mathbb{A}_{J_k}}(\by)}{\ell(\bone)}
    \int_{z \in \reals} f_{\bU^{(k)}} \rbr{\rbr{y_j-\ln a_{jk}+z}_{j \in J_k}} e^z \diff z.
    \label{equa:densityofmixturemgp}
    \end{equation}
\end{proposition}

Proposition~\ref{propo:densityofmixturemgp} will be helpful in the next subsection when we let each of the factors $\bZ^{(1)},\ldots,\bZ^{(r)}$ follow a parametric model.

\subsection{Parametric examples}
\label{sec:examples}

We construct two parametric families of the mixture model $\bM$ in Definition~\ref{def:mixture} with coefficient matrix $A$ by choosing particular forms for the factors $\bZ^{(1)},\ldots,\bZ^{(r)}$: the logistic model and the \HR~model. In each case, we compute the stdf $\ell$, the generator $\bU$ and the density $h_{\bY}$ of the associated mgp distribution. 

\subsubsection{Mixture logistic model}
\label{subsec:mixturelogisticmodel}

For $k \in R$, recall $J_k$ in \eqref{eq:Jk} and suppose that the stdf associated with $\bZ_{J_k}^{(k)}$ is
\begin{equation}
\label{eq:stdflog}
    \ell_{J_k}^{(k)}(\by)
    = \rbr{\sum_{j \in J_k} y_j^{1/\alpha_k}}^{\alpha_k},
\end{equation}
for $\by \in [0, \infty)^{J_k}$, where $0 < \alpha_k < 1$. The only extreme direction of $\bZ_{J_k}^{(k)}$ is $J_k$, so that, following Definition~\ref{def:mixture} and Remark~\ref{rem:ZkJk}, the \emph{mixture logistic model} $\bM$ is well-defined. 
Let $G_{\bM}$ denote the cdf of $\bM$; note that $G_{\bM}$ is an mev distribution with unit-Fréchet margins. We compute its stdf using Theorem~\ref{stephanson}.

\begin{corollary}
The stdf of the mixture logistic model $\bM$ is
\begin{equation}
\label{stdfmixturelogisticmodel1}
    \ell(\by) = \sum_{k \in R} \cbr{ \sum_{j  \in J_k} \rbr{a_{jk}y_j}^{1/\alpha_k} }^{\alpha_k},
    \qquad \by \in [0, \infty)^d.
\end{equation}
\end{corollary}

\begin{remark}
    In case the sets $J_1,\ldots,J_r$ are all different, the stdf $\ell$ in \eqref{stdfmixturelogisticmodel1} coincides with the one of the asymmetric logistic model introduced in \citet{T1990}.
\end{remark}

Next we construct a $\bU$-generator of the mgp distribution $H=\GP(\tilde{G}_{\bM})$, where $\tilde{G}_{\bM}$ is the cdf of $\rbr{\ln M_1,\ldots, \ln M_d}$, an mev distribution with Gumbel margins. Let $\bU^{(1)},\ldots,\bU^{(r)}$ be random vectors in \amtwo{$\reals^{J_1},\ldots,\reals^{J_r}$}, respectively, such that  
\begin{equation}
    \begin{split}
    &\bU^{(1)},\ldots,\bU^{(r)} \text{ are independent, } \\ 
    &U_{j}^{(k)} \eqd \ln\left( \frac{ X_{j}^{\alpha_k}} {\Gamma(1-\alpha_k)} \right), \qquad k \in R, \ j \in J_k,
    \end{split}
    \label{equa:generatormixturemodel} 
\end{equation}
where $X_{1},\ldots,X_{d}$ are independent unit-Fr\'echet random variables and where $\Gamma$ denotes the gamma function.

\begin{proposition}    \label{propo:generatormixturelogisticmodel}
      Let $\bo{m} \in (0,1)^{r}$ be such that $\sum_{k=1}^r m_k=1$. The distribution~\eqref{equa:generalgenerator} with coefficient matrix $A$, generators $\bU^{(1)},\ldots,\bU^{(r)}$ in \eqref{equa:generatormixturemodel} and mass vector $\bo{m}$ is the one of a $\bU$-generator of the mgp distribution $H = \GP(\tilde{G}_{\bM})$ for the mixture logistic model $\bM$. 
\end{proposition}

Let $\bY \sim H$ be an mgp random vector associated with the mixture logistic model $\bM$. The density of $\bY$ with respect to $\upsilon$ in \eqref{eq:upsilon} follows from Proposition~\ref{propo:densityofmixturemgp}.

\begin{proposition}
\label{propo:density}
    The mgp random vector $\bY \sim H$ associated with the mixture logistic model $\bM$ in \eqref{eq:stdflog} is absolutely continuous with respect to the measure $\upsilon$ with density 
    \begin{equation}
   \label{equa:densitymixturelogistic}
        h_{\bY}(\by)=\sum_{k \in R} \indicator_{\mathbb{A}_{J_k}} (\by)\point \frac{\rbr{1/\alpha_k}^{\lvert J_k \rvert -1} \Gamma\rbr{\lvert J_k \rvert - \alpha_k} \prod_{j \in J_K} \cbr{ \rbr{a_{jk}  \ee^{-y_j}}^{1/\alpha_k}  } }{ \ell(\bone) \,
        \Gamma\rbr{1 - \alpha_k}\cbr{\sum_{j \in J_k}  \rbr{ a_{jk}  \ee^{-y_j}}^{1/\alpha_k}  }^{\lvert J_k \rvert-\alpha_k}   } , 
    \end{equation}
    for $\by \in [-\infty, \infty)^d \setminus [-\infty, 0]^d$, with $\ell$ as in \eqref{stdfmixturelogisticmodel1}. 
\end{proposition}

\subsubsection{Mixture \HR~model}
\label{subsec:mixturehrmodel}

For $k \in R$, recall $J_k$ in \eqref{eq:Jk} and suppose that $\bZ_{J_k}^{(k)}$ has \HR~mev distribution \citep{hueslerReiss1989} with variogram matrix\footnote{A matrix $\Gamma \in \reals^{d \times d}$ is a variogram matrix if it is symmetric, has zero diagonal, and satisfies $\bo{v}^{\top} \Gamma \bo{v}<0$ for all $\bzero  \neq \bo{v}^\top \perp \bone$. To each positive definite covariance matrix $\Sigma \in \reals^{d \times d}$ is associated a variogram matrix $\Gamma \in \reals^{d \times d}$ via $\Gamma_{st}=\Sigma_{ss}+\Sigma_{tt}-2 \Sigma_{st}$ for $s,t \in D$. } $\Gamma^{(k)} \in \reals^{J_k \times J_k}$; see, e.g., \citet{huser2013composite}.
If $J_k$ is a singleton, the stdf $\ell_{J_k}^{(k)}$ of $\bZ_{J_k}^{(k)}$ reduces to the identity function on $[0, \infty)$. Otherwise,
\[ \ell_{J_k}^{(k)}(\by)=
\sum_{ j \in J_k} y_j \, \Phi_{\lvert J_k \rvert-1}\rbr{\eta^{j,(k)}\rbr{\by}; \Sigma^{j,(k)} }, \qquad \by \in [0,\infty)^{J_k},
\]
where for $j \in J_k$, 
$\eta^{j,(k)}\rbr{\by}$ is equal to $\cbr{\ln(y_j/y_s)+\Gamma^{(k)}_{js}/2}_{s \in J_k; s \neq j}$ with $y_j/0=\infty$ if $y_j>0$ and $\infty$ if $y_j=0$, 
where the matrix $\Sigma^{j,(k)}$ is defined by
\begin{equation*}
        \Sigma^{j,(k)}: =\frac{1}{2} \cbr{\Gamma_{js}^{(k)}+\Gamma_{jt}^{(k)}-\Gamma_{st}^{(k)}}_{s,t \in J_k \setminus \{j\}} \in \reals^{(J_k \setminus \{j\}) \times (J_k \setminus \{j\})}
        \label{equa:matrix} 
\end{equation*}
and finally, where $\Phi_{m}( \point;\Sigma)$ is the cdf of the centered $m$-variate normal distribution function with covariance matrix $\Sigma$.
The only extreme direction of $\bZ_{J_k}^{(k)}$ is $J_k$ as required in Definition~\ref{def:mixture} and Remark~\ref{rem:ZkJk}. The resulting $\bM$ in \eqref{mixturemodel} is called the \emph{mixture \HR~model}. Using Theorem~\ref{stephanson}, we compute the stdf of its cdf $G_{\bM}$.

\begin{proposition}
The stdf $\ell$ of $G_{\bM}$ is
\begin{equation}
\label{stdfmixtureHR}
    \ell(\by) = \sum_{k \in R} \cbr{ \sum_{ j \in J_k } a_{jk} y_j \, \Phi_{\lvert J_k \rvert-1}\rbr{\eta^{j,(k)}\rbr{\rbr{a_{jk}y_j}_{j \in J_k}}; \Sigma^{j,(k)} } },
    \qquad \by \in [0, \infty)^d.
\end{equation}
\end{proposition}

To construct a generator of the mgp distribution $H=\GP(\tilde{G}_{\bM})$ of the mixture \HR{} model, let $k \in R$ and let $\Tilde{\Sigma}^{(k)} \in \reals^{J_k \times J_k}$ be a positive definite covariance matrix with variogram matrix $\Gamma^{(k)}$; see, e.g., \citet{segers2020discussion} for a possible construction. Let $\bU^{(1)},\ldots,\bU^{(r)}$ be random vectors in $\reals^{J_1},\ldots,\reals^{J_r}$, respectively, such that  
\begin{equation}
	\begin{split}
          &\bU^{(1)},\ldots,\bU^{(r)} \text{ are independent, } \\
          &\bU^{(k)} \sim \mathcal{N}_{|J_k|} \left( 
          	-\tfrac{1}{2} \diag \Tilde{\Sigma}^{(k)}, \,
          	\Tilde{\Sigma}^{(k)} 
          	\right),
          	\qquad k \in R.
    \end{split}
\label{equa:generatormixtureHRmodel} 
\end{equation}
    
\begin{proposition}\label{propo:generatormixtureHRmodel}
      Let $\bo{m} \in (0,1)^{r}$ be such that $\sum_{k=1}^r m_k=1$. The distribution in \eqref{equa:generalgenerator} with coefficient matrix $A$, generators $\bU^{(1)},\ldots,\bU^{(r)}$ in \eqref{equa:generatormixtureHRmodel} and mass vector $\bo{m}$ is the one of a $\bU$-generator of the mgp distribution $H$ of the mixture \HR{} model. 
\end{proposition}

Let $\bY \sim H$ be an mgp random vector for $H = \GP(\tilde{G}_{\bM})$ and $\bM$ following the mixture \HR~model. The following proposition gives the density of $\bY$ with respect to the measure $\upsilon$ in \eqref{eq:upsilon}. Recall the sets $\mathbb{A}_J$ in \eqref{eq:AJ}.

\begin{proposition}
\label{Propo:densityHR}
    The mgp random vector $\bY \sim H$ associated with the mixture \HR{} model with variogram matrices $\Gamma^{(k)}$ for $k \in R$ and coefficient matrix $A$ is absolutely continuous with respect to the measure $\upsilon$ with density
\begin{align}
\label{equa:densitymixturehr}
h_{\bY}(\by)
=\sum_{k \in R} \indicator_{\mathbb{A}_{J_k}} (\by)  \point  \frac{\lambda\rbr{\by_{J_k}-\rbr{\ln a_{jk}}_{j \in J_k}; \Gamma^{(k)}}}{\ell(\bone)},
\end{align}
for $\by \in [-\infty, \infty)^d \setminus [-\infty, 0]^d$, with $\ell$ as in \eqref{stdfmixtureHR} and
where for $k \in R$, the function $\lambda( \point;\Gamma^{(k)}): \reals^{J_k} \to [0,\infty)$ is the density of the exponent measure of a \HR~mev distribution with Gumbel margins.
\end{proposition}

\begin{remark}
An expression for $\lambda( \point;\Gamma^{(k)})$ can be obtained for instance from \citet[Definition~3.1]{hentschel2022statistical}: writing $d(k) := \max(J_k)$, we have
\begin{multline*}
    \lambda \rbr{\bx; \Gamma^{(k)}}=\frac{\exp\rbr{-x_{d(k)}}}{ \sqrt{\rbr{2 \pi}^{\lvert J_k \rvert -1}} \det\rbr{\Sigma^{d(k),(k)}} } \\
    \cdot \exp\rbr{ - \frac{1}{2} \norma[\bigg]{\rbr{\bx_{J_k \setminus \cbr{d(k)}}}^{\top}-x_{d(k)} \bone_{\lvert J_k \rvert-1} -\rbr{\bo{\mu}^{d(k)}}^{\top} }^{2}_{\Theta^{d(k)}}}
\end{multline*}
with $\bone_m = (1,\ldots,1) \in \reals^{m}$ and
\[
    \Theta^{d(k)} 
    = \rbr{\Sigma^{d(k),(k)}}^{-1}, \qquad 
    \bo{\mu}^{d(k)}
    = \cbr{-\frac{1}{2}\Gamma_{sd(k)} }_{ s \in J_k \setminus \cbr{d(k)}}, 
\] 
and, finally, \amtwo{$\bnorm{\bo{g}}_{M}^2=\bo{g}^{\top} M \bo{g}$ for $M \in \reals^{p \times p}$ and $\bo{g} \in \reals^{p \times 1}$}.
In fact, replacing $d(k)$ by any other choice of index from $J_k$ would give the same function.
\end{remark}

\section{Simulation from mixture model mgp distribution}
\label{sec:simulation}

In Section~\ref{subsec:algo}, we propose a generic algorithm to simulate random samples from the mgp distribution associated to the mixture model in Definition~\ref{mixturemodel}. The details of the algorithm for the mixture logistic model and the mixture~\HR~model are worked out in Section~\ref{sec:simu:param}. 

\subsection{Generic simulation algorithm}
 \label{subsec:algo}
 
Recall the mixture model $\bM$ in Definition~\ref{def:mixture}. The goal of this section is to generate a random variable $\bY$ from the associated mgp distribution $H$ in~\eqref{equa:H-mgp} \amtwo{and with density in Proposition~\ref{propo:densityofmixturemgp}}.

Let $\bU$ with mixture distribution in~\eqref{equa:generalgenerator} be a $\bU$-generator of $H$ and let $\bT$ be the $\bT$-generator obtained from $\bU$ via the measure transformation in~\eqref{equa:TtoU}.
From~\eqref{ZtoT}, one can easily obtain $\bY$ from $\bT$ and an independent unit-exponential random variable. The remaining question is how to simulate random samples from the law of $\bT$. The following lemma expresses the random vector $\bT$ as a mixture over $k \in R$ of random vectors $\bT^{(k)}$ related to generators $\bU^{(k)}$ in \eqref{equa:U^(k)}. Recall the notation introduced in the beginning of Section~\ref{sec:background}, in particular the sets $\mathbb{A}_J$ and the projection $\pi_J$ in \eqref{eq:AJ}.

\begin{proposition}
\label{lemma:conditionalT's}
The $\bT$-generator in \eqref{equa:TtoU} resulting from the mixture generator $\bU$ in \eqref{equa:generalgenerator} satisfies 
\[
    \pr[\bT \in \point] = \sum_{k \in R} w_{k} \pr\sbr{ \bT^{(k)} \in \pi_{J_k} \rbr{\mathbb{A}_{J_k} \cap \point} },
\]
with $J_k$ as in \eqref{eq:Jk} and where the distribution of the $|J_k|$-dimensional random vector $\bT^{(k)}$ is
\begin{equation}
\pr\left[ \bT^{(k)}\in \point \right]
    = \frac{\expec\sbr{\exp\cbr{\max \bP^{(k)}} \indicator\cbr{ \bP^{(k)} \in \point}}}%
    {\expec\sbr{\exp\cbr{\max \bP^{(k)}}}}, \qquad
w_{k}
=\frac{ \ell_{J_k}^{(k)} \rbr{ \rbr{a_{jk}}_{j \in J_k} }  }{\ell\rbr{\bone}},
      \label{equa:lawT^k}
\end{equation}
where $\bP^{(k)}=\bU^{(k)}+\rbr{\ln\rbr{a_{jk}/m_k}}_{j \in J_k}$ with $\bU^{(k)}$ in \eqref{equa:U^(k)}, $\ell_{J_k}^{(k)}$ in \eqref{equa:k-thstdf} and $\ell$ in \eqref{stdfgeneralmodel}. Moreover, if $\bU^{(k)}$ has density $f_{\bU^{(k)}}$ with respect to $\lambda_{|J_k|}$, then $\bT^{(k)}$ has density 
\begin{equation}
	\label{equa:density}
	f_{\bT^{(k)}}(\bt) =c^{(k)} \, g_k(\bt) \, \sum_{j \in J_k} n_{j,k} \, q_{j,k}(\bt), \qquad \bt \in \reals^{J_k},
\end{equation}
with $g_k(t) = \ee^{\max(\bt)} / \sum_{j \in J_k} \ee^{t_j}$ and
\begin{equation*}
	c^{(k)} := \frac{\sum_{j \in J_k}a_{jk}}{\ell_{J_k}^{(k)} \rbr{ (a_{jk})_{j \in J_k} }},    
\end{equation*}
and, for $j \in J_k$,
\begin{align}
	n_{j,k}&:=\frac{a_{jk}}{\sum_{i \in J_k} a_{ik}}, \label{equa:njkqjk}\\
	q_{j,k}\rbr{\bt}
	&:=
	\frac{\ee^{t_j} f_{\bU^{(k)}} \rbr{\bt -\rbr{\ln \rbr{a_{ik}/m_k}}_{i \in J_k}} }{a_{jk}/m_k}.
	\label{equa:njkqjk2}
\end{align}
\end{proposition}

In view of Proposition~\ref{lemma:conditionalT's}, we can simulate random samples from the distribution of $\bT$ once we can do so for the distributions of $\bo{T}^{(1)},\ldots,\bo{T}^{(r)}$.
The representation \eqref{equa:density} features a mixture over exponentially tilted versions of the density of $\bU^{(k)}$ times a uniformly bounded function. This opens the door to simulation via a rejection sampling algorithm. 
As a consequence, we can simulate random samples from the distribution of $\bT^{(k)}$ in \eqref{equa:lawT^k} via rejection sampling and random vectors with densities $q_{j,k}$ for $j \in J_k$. The algorithm hence requires that $\bU^{(k)}$ in \eqref{equa:U^(k)} is absolutely continuous and that one can simulate from $q_{j,k}$ for all $k \in R$ and $j \in J_k$. The procedure is summarized in Algorithm~\ref{alg:cap}.

\begin{algorithm}[ht]
\caption{Simulation of an mgp random vector associated with a mixture model }
\label{alg:cap}
    \hspace*{\algorithmicindent} \textbf{Input:} integers $d \geq 1$ and $r \geq 1$, a matrix $A \in [0,1]^{d \times r}$ as in Definition~\ref{def:mixture} and distributions $\law\rbr{\bU^{(1)}},\ldots, \law\rbr{\bU^{(r)}}$ satisfying \eqref{equa:U^(k)}. \\
    \hspace*{\algorithmicindent} \textbf{Output:} an mgp vector $\bY$ with distribution $H$ generated by $\bU$ in \eqref{equa:generalgenerator} with matrix $A$ and generators $\bU^{(1)},\ldots,\bU^{(r)}$.
\begin{algorithmic}[1]
\State Define $\operatorname{sign}[[k]]:=J_k$ in \eqref{eq:Jk} for $k \in \cbr{1,\ldots,r}$. \Comment{sign is the list of signatures of $A$.}
\State Define the $d$-dimensional vector $\bT:=(-\infty,\ldots,-\infty)$.
\State Simulate $b$ from $\cbr{1,\ldots,r}$ such that $\pr\rbr{b=k}=w_{k}$ in \eqref{equa:lawT^k} for $k \in \cbr{1,\ldots,r}$.
\State accept:=FALSE.
\While {accept=FALSE}
\State Simulate $a$ from $\operatorname{sign}[[b]]$ such that $\pr\rbr{a=j}=n_{j,b}$ in \eqref{equa:njkqjk} for $j \in  \operatorname{sign}[[b]]$. 
\State Simulate $\bo{Q}\sim q_{a,b}$ in \eqref{equa:njkqjk2}.
\State Simulate $U_{0} \sim \operatorname{Unif}(0,1)$. 
\If {$U_0\leq  \exp\rbr{\max \bo{Q}} / \sum_{j} \exp\rbr{Q_j}$}
accept:=TRUE.
\EndIf
\EndWhile
\State Update $\bT_{\operatorname{sign}[[b]]}:=\bo{Q}$.
\State Simulate a unit exponential random variable $E$.
\State Define $\bY:=\bT-\max(\bT)+E$.
\State \Return($\bY$).
\end{algorithmic}
\end{algorithm}

The complexity of Algorithm~\ref{alg:cap} is driven by the expected number of times one has to simulate from densities $\rbr{q_{j,k}}_{k \in R, j \in J_k}$ in~\eqref{equa:njkqjk2} and is given by $d/\ell(\bone)$ with $\ell$ in \eqref{stdfgeneralmodel}; see Appendix~\ref{sec:proofs5} for the proof. \amtwo{This is equal to the complexity of the algorithm in \citet[Subsection~5.4]{engelke2020graphical} based on extremal functions $\bo{\mathcal{U}}^{(1)},\ldots,\bo{\mathcal{U}}^{(d)}$ 
as defined in  \citet{dombry2016exact} and characterized in \citet[Lemma~2]{engelke2020graphical} in terms of the mgp random vector $\bY$} by 
\[
	\bo{\mathcal{U}}^{(j)}
	\eqd \exp \rbr{\bY - Y_{j}} \mid Y_j>0, \qquad j \in D.
\] 
The following proposition connects these extremal functions to the densities $q_{j,k}$ in \eqref{equa:njkqjk2}.

\begin{proposition}
\label{propo:extremalfunction}
    If $\bU$ in~\eqref{equa:generalgenerator} has density $f_{\bU}$ with respect to $\upsilon$, then for $j \in D$, we have 
    \[
    \bo{\mathcal{U}}^{(j)} \eqd \exp\rbr{\bQ^{(j)}-Q^{(j)}_j},
    \]
where $\bQ^{(j)}$ is a $d$-variate random vector with density $\ee^{t_j}f_{\bU}(\bt)$ with respect to $\upsilon$.    
\end{proposition}


\js{There are thus two possible ways to simulate $\bT$:
\begin{itemize}
\item Method (a): simulate $\bT$ as a mixture over $\bT^{(1)},\ldots,\bT^{(k)}$ in~\eqref{equa:lawT^k}, where we simulate $\bT^{(k)}$ from $\bU^{(k)}$ via \eqref{eq:fV2fT} and rejection sampling.
\item Method (b): Simulate $\bT$ from $\bU$ via \eqref{eq:fV2fT} and rejection sampling, where we simulate from $\bU$ as a mixture over $\bU^{1}, \ldots, \bU^{(r)}$ via \eqref{equa:generalgenerator}.
\end{itemize}
Algorithm~\ref{alg:cap} corresponds to method~(a). The method from \citet{engelke2020graphical} rather corresponds to method~(b). In case $r=1$, they are the same. In case $r > 1$, can we compare both methods in terms of ease of application, speed, \ldots?}
\ak{I don't think we need to explain how to simulate $\bo{\mathcal{U}}^{(j)}$ via $\bU$ or vice versa. Johan's explanation above could be added instead.}

\subsection{Parametric examples}
\label{sec:simu:param}

We use the results in Section~\ref{subsec:algo} to draw random samples from the mgp distributions associated with the mixture logistic and \HR{} models from Section~\ref{sec:examples}.
 
\subsubsection{Mixture logistic model}

We want to simulate the mgp random vector $\bY$ associated with the mixture logistic model defined in Section~\ref{subsec:mixturelogisticmodel}. Let $\bU$ be the mixture generator in~\eqref{equa:generalgenerator} with matrix $A$ and generators $\bU^{(1)},\ldots,\bU^{(r)}$ in \eqref{equa:generatormixturemodel} with parameters $\alpha_1,\ldots,\alpha_r$. The $\lvert J_k \rvert$-variate random vector $\bU^{(k)}$ has independent Gumbel components and is therefore absolutely continuous with respect to $\lambda_{\lvert J_k \rvert}$. Using Algorithm~\ref{alg:cap}, it remains to generate random samples from the densities $q_{j,k}$ ($j \in J_k$) defined for $\bt=\rbr{t_i}_{i \in J_k} \in \reals^{J_k}$ by
    \begin{equation}
         q_{j,k}(\bt)
 =\frac{\ee^{t_{j}} f_{U_{j}^{(k)}}\rbr{t_{j}-\ln \rbr{a_{jk}/m_k}}}{a_{jk}/m_k} \cdot \prod_{i \in J_k \setminus \cbr{j}} f_{U_i^{(k)}}\rbr{t_i-\ln \rbr{a_{ik}/m_k}}.
 \label{equa:densityqj0}
    \end{equation}
\begin{proposition}
\label{propo:simulationqlogostic}
Let $k \in R$ and $j \in J_k$.
   Let $\bo{\psi}=\rbr{\psi_i}_{i \in J_k}$ have density $q_{j,k}$ in \eqref{equa:densityqj0}. The distribution of $\bo{\psi}$ is that of a random vector with independent components, such that 
    \[ \psi_i
      \eqd \begin{dcases}
		U_i^{(k)} +\ln \rbr{a_{ik}/m_k}, & \text{if $i \in J_k \setminus \cbr{j}$,} \\
		-\alpha_k\ln N+\ln\rbr{a_{jk}/\rbr{m_k \Gamma(1-\alpha_k)} },  & \text{if $i=j$,}
	\end{dcases}
    \]
where $N$ is a Gamma random variable, independent from $\bU^{(k)}$, with shape and rate parameters $1-\alpha_k$ and $1$ respectively.
\end{proposition}

\subsubsection{Mixture \HR~model}

In the same way, we want to draw random samples from the mgp random vector $\bY$ associated with the mixture \HR~model in Section~\ref{subsec:mixturehrmodel}. Let $\bU$ be the mixture generator in~\eqref{equa:generalgenerator} with matrix $A$ and generators $\bU^{(1)},\ldots,\bU^{(r)}$ in \eqref{equa:generatormixtureHRmodel}. Let $k \in R$. The $\lvert J_k \rvert$-variate random vector $\bU^{(k)}$ has normal density $f_{\bU^{(k)}}$. \amtwo{Using Algorithm~\ref{alg:cap}, it remains to simulate from the densities $q_{j,k}$ ($j \in J_k$) in Equation~\eqref{equa:densityqj0}, with the difference that the generators $\bU^{(1)},\ldots,\bU^{(r)}$ are defined via \eqref{equa:generatormixtureHRmodel} instead of \eqref{equa:generatormixturemodel}.}

\begin{proposition}
\label{propo:pdfhrq}
Let $j \in J_k$. Let $\bo{\zeta}=\rbr{\zeta_j}_{i \in J_k}$ have density $q_{j,k}$ \amtwo{in~\eqref{equa:densityqj0} with generators $\bU^{(1)},\ldots,\bU^{(r)}$ as in \eqref{equa:generatormixtureHRmodel}.} With $\Tilde{\Sigma}^{(k)}$ as in \eqref{equa:generatormixtureHRmodel}, we have
    \[
    \bo{\zeta} \sim \Normal\rbr{ \rbr{\ln \rbr{a_{ik}/m_k}-\Tilde{\Sigma}_{ii}^{(k)}/2+\Tilde{\Sigma}^{(k)}_{ij}}_{i \in J_k}, \Tilde{\Sigma}^{(k)} }.
    \]    
\end{proposition}


\subsubsection{Simulation example}
\amtwo{The results in this section can be reproduced using the code available on \url{https://github.com/AnasMourahib/MGPD-simulation}.} We use Algorithm~\ref{alg:cap} with inputs $d=r=3$, triangular coefficient matrix 
\[A=
\begin{pmatrix}
1 & 0 & 0\\
1/2 & 1/2 & 0 \\
1/3 & 1/3 & 1/3
\end{pmatrix},
\]
and two different specifications of the generators
$\bU^{(1)},\bU^{(2)},\bU^{(3)}$: first, those in \eqref{equa:generatormixturemodel} with dependence parameters $\alpha_1=\alpha_2=\alpha_3=0.5$ to simulate the mgp random vector $\bY$ associated with the mixture logistic model;  second, those in \eqref{equa:generatormixtureHRmodel} based on a variogram matrix with all off-diagonal elements equal to $1.38$, to simulate the mgp random vector $\bY$ associated with the mixture \HR~model. Figure~\ref{fig:mgp} displays the pairwise scatter plots of a sample of size $100$ from the transformed mgp random vector 
\begin{equation}
\label{equa:Z}
\bZ=4\cbr{\exp\rbr{\bY/4 }-1},    
\end{equation}
associated with the mixture logistic and \HR~models in panels~(a) and~(b), respectively. The Box--Cox type transformation in \eqref{equa:Z} ensures that the vector of lower endpoints of $\bZ$ is $\bo{\eta}=(-4,-4,-4)$ and facilitates visualization. 
\jsch{By Proposition~\ref{propo:extremedirections}, the matrix $A$ produces a mixture model with three distinct extreme directions $J \subset D$, namely $\cbr{1,2,3}$, $\cbr{2,3}$, and $\cbr{3}$. This is visible in Figure~\ref{fig:mgp} by a concentration of points on the three corresponding subfaces $\mathbb{A}^{-\bo{4}}_{J}= \cbr{\bx \in [-4,\infty)^3: \ x_j>-4 \text{ iff } j \in J }$.}



\begin{figure}[ht]
    \centering
       \subfloat[\centering ]{{\includegraphics[width=0.48\textwidth]{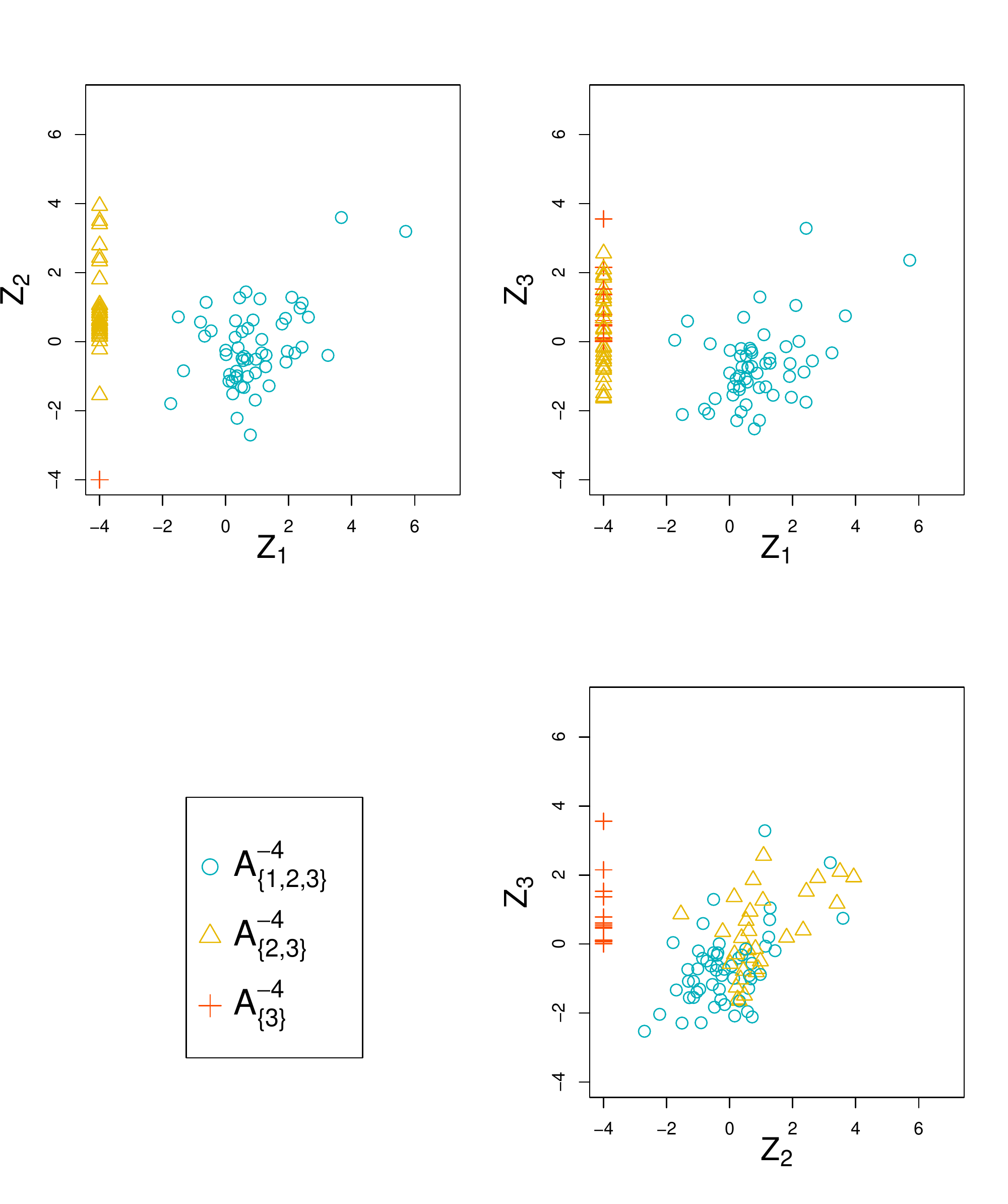} }}%
    \subfloat[\centering ]{{\includegraphics[width=0.48\textwidth]{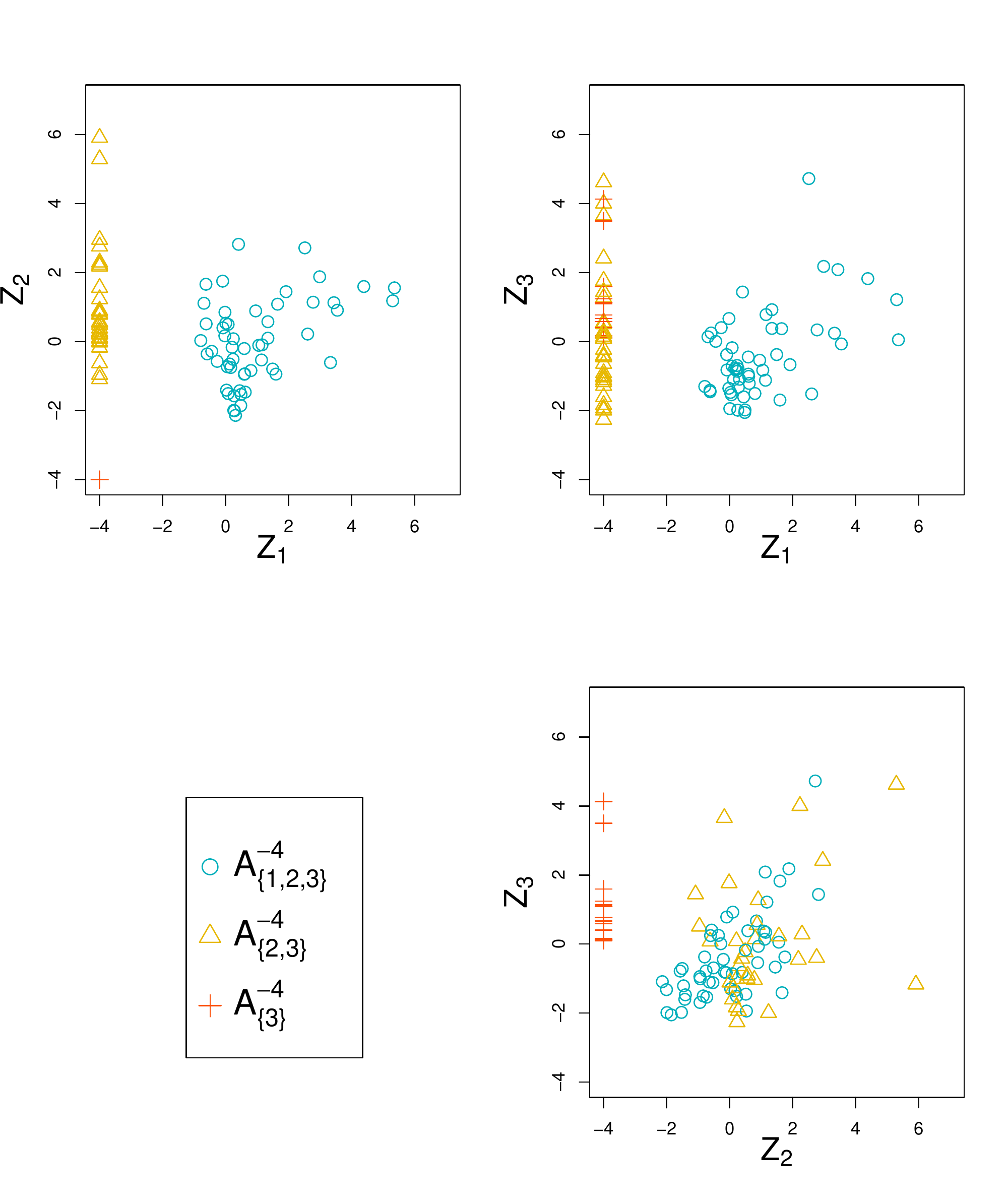} }}%
    \caption{Sample of size $100$ from the transformed mgp random vector $\bZ$ in \eqref{equa:Z} associated with the mixture logistic model in (a) and the mixture \HR~model in (b).} %
    \label{fig:mgp}%
\end{figure}

\amch{To assess Algorithm~\ref{alg:cap}, Table~\ref{table} compares the empirical proportions of points in the three sub-faces $\mathbb{A}_{J}^{-4}$ in a sample of size $1000$ with the corresponding true probabilities in Eq.~\eqref{equa:lawT^k} for the logistic and \HR\ mixture models.}

\begin{table}[htbp]
  \centering
  \caption{Empirical proportions and true probabilities of lying in each sub-face $\mathbb{A}_{J}^{-4}$ for the extreme directions $J = \cbr{1,2,3}$, $\cbr{2,3}$ and $\cbr{3}$ based on samples of size $1000$ drawn from the logistic and \HR\ mixture models by means of Algorithm~\ref{alg:cap}.}
  \label{tab:example}
  \begin{tabular}{c|cc|cc} 
    \toprule
    \multicolumn{1}{c}{} & \multicolumn{2}{|c|}{\textbf{mixture logistic}} & \multicolumn{2}{|c}{\textbf{mixture \HR}} \\ 
    \midrule
    extreme direction $J$ & \color{white}a \color{black} empirical & true & empirical & true \\ 
    \midrule
     $\cbr{1,2,3}$ & \color{white}a \color{black} 0.559 & 0.555 & 0.560 & 0.553 \\
     $\cbr{2,3}$ &\color{white}a \color{black} 0.297 & 0.286 & 0.291 & 0.288 \\
     $\cbr{3}$ & \color{white}a \color{black}0.144 & 0.158 & 0.148 & 0.157 \\
    \bottomrule
  \end{tabular}
  \label{table}
\end{table}

\section{Conclusion}
\label{sec:conclusion}

The extreme directions of a random vector refer to those groups of variables that can be large simultaneously while those outside the group are not. As such scenarios are likely to occur in high dimensions or in spatial contexts, we have studied mgp models featuring arbitrary collections of extreme directions. Furthermore, we have provided a general stochastic representation of mev distributions with arbitrary extreme directions and we have computed their corresponding mgp distributions. The construction involved a mixture specification extending the max-linear model by including multivariate factors. The density of such an mgp distribution---and thus of the exponent measure of the mev distribution---with respect to an appropriate dominating measure has been expressed in terms of those of its generators. In addition, we have provided a simulation algorithm for drawing samples from such mgp distributions. The formulas have been made explicit for two parametric models, extending the classical logistic and Hüsler–Reiss mev and mgp distributions to families with arbitrary extreme directions.

Our work raises intriguing questions for further research. The foremost among these is the challenge of statistical inference for mgp mixture models based on multivariate threshold excesses. This multifaceted problem encompasses parameter estimation, model selection, and the incorporation of sparsity constraints.

Additionally, our research opens doors to the application of mgp mixture models in the realm of graphical models for extremes. Whereas the foundations laid by \citet{engelke2020graphical} involve a single extreme direction, the extension offered in \citet{engelke2022graphical} incorporates general exponent measures, enabling a more comprehensive understanding of extreme events in diverse contexts.

\appendix

\section{Proofs of Section~\ref{sec:Directions}}
\begin{proof}[Proof of Proposition~\ref{propo:angularmeasureinterpretation}]
The monotone convergence theorem and homogeneity of $\mu$ yield
\begin{align*}
 \mu(\EEJ)>0
 &\iff \forall r > 0 : \; \mu\rbr{\cbr{\bx \in \EEJ: \; \| \bx \|>r }}>0\\
 &\iff \mu\rbr{\cbr{\bx \in \EEJ: \; \| \bx \|>1 }}>0.
\end{align*}  
Since $\mu\rbr{\cbr{\bx \in \EEJ: \; \| \bx \|>1 }} =\phi(\sphereJ)$, the conclusion follows.
\end{proof}

\begin{proof}[Proof of Lemma~\ref{lemma:halfrectangles}]
Let $\eps>0$. The $\mu$-continuity of $\mathbb{R}_{1,J}^{\eps}$ gives   
\begin{equation}
\label{equa:limitmucontinuity}
    n \pr\left[ \bX \in \rectangle \right]
    = n \pr\left[ n^{-1} \bX \in \mathbb{R}_{1,J}^{\eps} \right]
    \to \mu \left( \mathbb{R}_{1,J}^{\eps} \right), \qquad n \to \infty.    
\end{equation}
By the dominated convergence theorem, letting $\eps$ tend to zero yields
\begin{equation}
\label{equa:dominatedconv}
    \lim_{\eps \to 0} \mu \left( \mathbb{R}_{1,J}^{\eps} \right)
    = \mu \rbr{ \cbr{\bx \in \EEJ: \; \| \bx \|>1 }}
    = \phi \left( \sphereJ \right),
\end{equation}
as $\cbr{\bx \in \EEJ: \; \| \bx \|>1 } = \cbr{ \bx \in \EE : \, \| \bx \| > 1, \bx / \| \bx \| \in \sphereJ }$.
Combining limits~\eqref{equa:limitmucontinuity} and~\eqref{equa:dominatedconv} gives \eqref{equation:halfrectangles}.
\end{proof}

\begin{proof}[Proof of Lemma~\ref{lemma:halfcones}]
Let $N \subset [0,\infty]$ be the set that contains those $\eps>0$ for which the boundary of $\{\bw \in \mathbb{S}: \; w_j>\eps \text{ iff } j \in J\}$ as a 
subset of $\mathbb{S}$ receives positive mass by $\phi$. Since $\phi$ is a finite measure, $N$ is at most countable. Restricting $\eps>0$ to $(0,\infty) \setminus N$  and applying the same reasoning as in the proof of Lemma~\ref{lemma:halfrectangles} yields \eqref{equality:halfcones}.
\end{proof}

\begin{proof}[Proof of Lemma~\ref{linkbetweenrepresntations}]
From the representation $\bY = \bS + E$ in \eqref{ZtoS}, we deduce
\begin{align*}
    \pr \sbr{ Y_j > -\infty \text{ iff } j \in J }
    &= \pr \sbr{ S_j + E > -\infty \text{ iff } j \in J } \\
    &= \pr \sbr{ S_j > -\infty \text{ iff } j \in J }.
\end{align*} 
The identity~\eqref{YtoS} follows. Next, combining $\bS=\bT-\max(\bT)$ with $-\infty<\max(\bT)<\infty$ almost surely yields \eqref{YtoT}. Finally, the random vector $\bU$ is linked to $\bT$ via \eqref{equa:TtoU}, so that
\[
    \pr \sbr{ T_j > -\infty \text{ iff } j \in J }
    = \frac{\expec\sbr{ \ee^{\max(\bU)} \indicator \cbr{ U_j > -\infty \text{ iff } j \in J }}}{\expec \sbr{ \ee^{\max( \bU)} } }.
\]
Since $\expec\sbr{ \ee^{\max(\bU)} \indicator \cbr{ U_j > -\infty \text{ iff } j \in J }} > 0$ if and only if $\pr \sbr{ U_j > -\infty \text{ iff } j \in J } > 0$, equivalence~\eqref{YtoU} follows. 
\end{proof}

\section{Proofs of Section~\ref{sec:sparsemgpds}}
\label{sec:proofs4}
\begin{proof}[Proof of Theorem~\ref{stephanson}]
Let $\bx \in [0,\infty)^d$. If $x_j=0$ for some $j \in D$, then $G^{n}_{\bM}(n\bx) = G_{\bM}(\bx) = 0$ for all $n \geq 1$. Suppose that $\bx \in (0,\infty)^d$. By \eqref{mixturemodel}, $G_{\bM}(\bx)$ is given by
\begin{align}
    &\pr\sbr{ \max_{k\in R}\cbr{a_{1k} Z_{1}^{(k)}} \leq x_1, \ldots, \max_{k\in R}\cbr{a_{dk} Z_{d}^{(k)} }\leq x_d } 
    =\pr\sbr{ \forall k \in R : \forall j \in J_k :  a_{jk}Z_{j}^{(k)} \leq x_j } \nonumber \\
    &=\prod_{k\in R} \pr\sbr{\forall j \in J_k : a_{jk}Z_{j}^{(k)} \leq x_j}
    =\prod_{k \in R} \exp\cbr{ -\ell_{J_k}^{(k)} \rbr{ \rbr{ a_{jk}/x_j }_{j \in J_k} } }\nonumber \\
    &=\exp\cbr{-\sum_{k \in R} \ell_{J_k}^{(k)} \rbr{ \rbr{ a_{jk}/x_j }_{j \in J_k} }   }.\label{equa:expressionGM}
\end{align}
Homogeneity of $\ell_{J_k}^{(k)}$ in $\eqref{equa:k-thstdf}$ for all $k \in R$ yields that $G^{n}_{\bM}(n \bx)=G(\bx)$ for all $n \geq 1$ so that $G_{\bM}$ is an mev distribution. Moreover, the constraints $\sum_{k=1}^{r} a_{jk}=1$ for all $j \in D$ in Definition~\ref{def:mixture} ensure that the margins of $G_{\bM}$ are unit-\FR. Finally, from \eqref{equa:defstdf} and \eqref{equa:expressionGM}, we deduce that the stdf of $G_{\bM}$ is given by \eqref{stdfgeneralmodel}.
\end{proof}

\begin{proof}[Proof of Proposition~\ref{propo:extremedirections}]
Recall $\mu$, the exponent measure of $G$. For $k \in R$, let $\mu_{J_k}^{(k)}$ be the exponent measure of the factor $\bZ_{J_k}^{(k)}$, let $\pi_{J_k} : \reals^d \to \reals^{J_k}$ be the projection $\bx \mapsto \bx_{J_k}$ and write $\mathbb{F}_{J_k}:=\cbr{\bz \in \EE: \; \forall j \notin J_k, z_j=0 }$. 
For $\bo{v} \in (0,1)^{m}$ and $A \subseteq [0,\infty)^{m}$, we write $\bo{v}\point A:=\cbr{\rbr{v_1a_1,\ldots,v_ma_m}: \ \bo{a}  \in A}$. Consider the measure 
\[
	\eta(B):=\sum_{k \in R} \mu_{J_k}^{(k)} \rbr{B^{(k)}},
\]
for Borel $B \subseteq \EE$, where $B^{(k)}:=\rbr{1/a_{jk}}_{j \in J_k} \point \pi_{J_k} \rbr{B \cap \mathbb{F}_{J_k}}$. For $\bx \in [0,\infty]^d \ \backslash  \cbr{\bzero}$, 
\begin{align*}
\eta\rbr{ \EE \setminus [\bzero,\bx] }
&=\sum_{k \in R} \mu_{J_k}^{(k)} \rbr{ \rbr{1/a_{jk}}_{j \in J_k} \point \pi_{J_k} \rbr{ \cbr{ \bz \in \EE: \exists j \in D, \; z_j > x_j } \cap \mathbb{F}_{J_k} }}\\
&=\sum_{k \in R} \mu_{J_k}^{(k)} \rbr{ \rbr{1/a_{jk}}_{j \in J_k} \point \pi_{J_k} \rbr{ \cbr{ \bz \in \EE: \exists j \in J_k, \; z_j > x_j } \cap \mathbb{F}_{J_k} }}\\
&=\sum_{k \in R} \mu_{J_k}^{(k)} \rbr{ \sbr{0,\rbr{x_j/a_{jk}}_{j \in J_k}}^c}
=\mu\rbr{ \EE \setminus [\bzero,\bx] },    
\end{align*}
where the last equality is a result of \eqref{stdfgeneralmodel}. By uniqueness of the exponent measure, we deduce that $\eta=\mu$. Let $J \subseteq D$ be non-empty and recall $\EE_J=\cbr{\bx \in \EE : x_j> 0 \text{ iff } j \in J}$. We find
\begin{equation*}
\mu \rbr{\EE_J}
=\sum_{k \in R} \mu_{J_k}^{(k)} \rbr{\rbr{1/a_{jk}}_{j \in J_k} \point \pi_{J_k} \rbr{\EE_J  \cap \mathbb{F}_{J_k}}}.
\end{equation*}
Let $k \in R$. We will show that the $k$-th term on the right-hand side is positive if and only if $J = J_k$. This will imply that $\mu(\EE_J) > 0$ if and only if $J \in \cbr{J_1,\ldots,J_r} = \mathcal{R}$, proving the proposition in view of Definition~\ref{def:Xdir}.

\begin{itemize}
\item If $J \setminus J_k \neq \emptyset$, then $\EE_J \cap \mathbb{F}_{J_k}=\emptyset$ and thus $\mu_{J_k}^{(k)} \rbr{\rbr{1/a_{jk}}_{j \in J_k} \point \pi_{J_k} \rbr{\EE_J  \cap \mathbb{F}_{J_k}}}=0$. 
\item Otherwise, if $J \subseteq J_k$, we get $\EE_J \cap \mathbb{F}_{J_k}=\EEJ$ and thus 
\[ \rbr{1/a_{jk}}_{j \in J_k} \point \pi_{J_k} \rbr{\EE_J  \cap \mathbb{F}_{J_k}}=\rbr{1/a_{jk}}_{j \in J_k} \point \pi_{J_k} \rbr{\EE_J}=\pi_{J_k} \rbr{\EE_J}. \]
Since $J_k$ is the only extreme direction of $\bZ_{J_k}^{(k)}$, we get $\mu_{J_k}^{(k)} \rbr{\pi_{J_k} \rbr{ \EE_J \cap \mathbb{F}_{J_k}} }>0$ if and only if $J=J_k$. \qedhere
\end{itemize}
\end{proof}

\begin{proof}[Proof of Proposition~\ref{propo-generator-general-model}]
For a non-empty set $J \subseteq D$, recall the set $\mathbb{A}_J$ defined in \eqref{eq:AJ} and the projection $\pi_{J} : \mathbb{A}_J \to \reals^{J}$ defined via $\bx \mapsto \bx_{J}$. Thus, for $J \subseteq D$ and $\bx \in \reals^{J}$, the point $\pi^{-1}_{J}(\bx)$ in $\mathbb{A}_J$ is defined by setting the components with indices $j \in J$ equal to $x_j$ and the components with indices $j \in D \setminus J$ equal to $-\infty$.

Let $g: [-\infty,\infty)^d \to [0,\infty)$ be a Borel measurable function. 
By linearity of the expectation and using standard reasoning from measure theory, we get, for $\bU$ with distribution \eqref{equa:generalgenerator},
\begin{equation}
\begin{aligned}
    &\expec\sbr{g\rbr{\bo{U}}}
    =\sum_{k \in R} m_k \expec\sbr{g\rbr{\pi_{J_k}^{-1} \rbr{ \bU^{(k)}+ \rbr{ \ln\rbr{a_{jk}/m_k} }_{j \in J_k} } }} .
    \label{equa:expectedofg}
    \end{aligned}
\end{equation}

First, let $j_0 \in D$. Equation~\eqref{equa:expectedofg} applied to the function $g :[-\infty,\infty)^d \to [0,\infty)$ defined via $g(\bx)=\ee^{x_{j_0}}$ gives 
\[
	\expec\sbr{\ee^{U_{j_0}}}
	= \sum_{k \in R : j_0 \in J_k}
	m_k \expec \sbr{\exp \cbr{ U^{(k)}_{j_0} + \ln \rbr{a_{j_0k} / m_k}}}
	= \sum_{k \in R : j_0 \in J_k} m_k \cdot \rbr{a_{j_0k} / m_k}
	= 1,  
\]
since $\expec \sbr{ \exp \rbr{ U^{(k)}_j } } = 1$ for all $j \in J_k$ and since $A$ has unit row sums.

Next, let $\by \in [0,\infty)^d$. Applying \eqref{equa:expectedofg} to the function $g(\bx)=\max\rbr{y_1\ee^{x_1},\ldots,y_d \ee^{x_d}}$ gives
\begin{align*}
	\expec \sbr{\max \rbr{\by \ee^{\bU}} }
	&= \sum_{k \in R} m_k \expec \sbr{
		\max \cbr{ \by \ee^{\pi_{J_k}^{-1} \rbr{ \bU^{(k)}+ \rbr{  \ln \rbr{a_{jk} /m_k} }_{j \in J_k} }  }  }
	}\\
    &=\sum_{ k \in R}  m_k \expec\sbr{ \max_{j \in J_k } \cbr{ y_j \frac{a_{jk}}{m_k} \ee^{U_{j}^{(k)} } } } \\
	&= \sum_{k \in R} \ell^{(k)}_{J_k}\rbr{\rbr{a_{jk}y_j}_{j \in J_k}}= \ell(\by),
\end{align*}
where we applied \eqref{stdfgeneralmodel} in the last step.
Hence, $\bU$ satisfies \eqref{eq:generator} and \eqref{equa:conditiononU} for the stdf $\ell$ of $G_{\bM}$. Thus, $\bU$ generates $H$.
\end{proof}
\begin{proof}[Proof of Theorem~\ref{theorem:general}]
Let $G$ be an mev distribution with unit-\FR~margins. Let $\mu$ be the exponent measure of $G$ and put $\mathcal{R}:=\cbr{ J \subseteq D: \; J \ne \emptyset, \; \mu\rbr{\EE_J}>0}$ where we recall $\EE=[0,\infty)^d \setminus \cbr{0}$ and $\EE_J=\cbr{\bx \in \EE: \; x_j >0 \text{ iff } j  \in J}$. We write $\mathcal{R}=\cbr{J_1,\ldots,J_r}$ with \amch{$2^{d}\ge r=\lvert \mathcal{R} \rvert \ge 1$} and $R=\cbr{1,\ldots,r}.$ Let $A = (a_{jk})_{j \in D;k \in R} \in [0,1]^{d \times r}$ be the coefficient matrix whose $(j,k)$-th coordinate is equal to 
\[
	a_{jk} =
	\begin{dcases}
		\mu \rbr{ \cbr{\by \in \EE_{J_k}: \; y_j>1} }, &
		\text{if $j \in J_k$,} \\
		0, & \text{if $j \in D \setminus J_k$.}				
	\end{dcases}
\]
For $\bo{v} \in (0,1)^{m}$ and $A \subseteq [0,\infty)^{m}$, we write $\bo{v}\point A:=\cbr{\rbr{v_1a_1,\ldots,v_ma_m}: \ \bo{a}  \in A}$. Fix $k \in R$ and consider the Borel measure $\mu_{J_k}^{(k)}$ on $(0,\infty)^{|J_k|}$ by
\[
	\mu_{J_k}^{(k)}\rbr{B}
	= \mu \rbr{\pi_{J_k}^{-1}\rbr{\rbr{a_{jk}}_{j \in J_k} \point B}},
\]
with $\pi_{J_k} : \EE_{J_k} \to (0,\infty)^{J_k}$ the natural coordinate projection. Let $\bo{Z}^{(k)}_{J_k}$ be a random vector on  $\reals^{J_k}$ with distribution 
\[
	\pr\sbr{\bo{Z}^{(k)}_{J_k} \leq \bo{x}}
	= \exp \cbr{-\mu_{J_k}^{(k)} \rbr{ [\bzero,\bx]^c } }, \qquad \bx \in [0,\infty]^{J_k} \setminus \cbr{\bzero}.
\]
The construction of $\mu_{J_k}^{(k)}$ and $\rbr{a_{jk}}_{j \in J_k}$ as above ensures that $\bZ_{J_k}^{(k)}$ has an mev distribution with unit-Fréchet margins and single extreme direction $J_k$. Following Definition~\ref{def:mixture} and Remark~\ref{rem:ZkJk}, consider the mixture random vector $\bM$ with coefficient matrix $A$ and factors $\bZ_{J_1}^{(1)},\ldots,\bZ_{J_r}^{(r)}$. Using Theorem~\ref{stephanson}, the random vector $\bo{M}$ follows an mev distribution $G_{\bM}$ with unit-\FR~margins and exponent measure $\tilde\mu$ satisfying for $\bx \in \EE$,
\begin{align*}
    &\tilde\mu\rbr{[\bzero,\bx]^c}
    =\sum_{k \in R} \mu_{J_k}^{(k)} \rbr{ \sbr{0,\rbr{x_j/a_{jk}}_{j \in J_k}}^c}
    =\sum_{k\in R} \mu_{J_k}^{(k)} \rbr{ \rbr{1/a_{jk}}_{j \in J_k} \point  \sbr{0,\bx_{J_k}}^c}\\
    &=\sum_{k \in R} \mu \rbr{\pi_{J_k}^{-1} \rbr{\rbr{a_{jk}}_{j \in J_k} \point \rbr{1/a_{jk}}_{j \in J_k} \point  \sbr{0,\bx_{J_k}}^c }}
    = \sum_{k \in R} \mu \rbr{\pi_{J_k}^{-1} \rbr{[\bzero,\bx_{J_k}]^c}}
    =\mu \rbr{ [\bzero,\bx]^c}.
\end{align*}
The last equality follows from $\pi_{J_k}^{-1} \rbr{[\bzero,\bx_{J_k}]^c} = \EE_{J_k} \setminus [\bzero,\bx]$ since the sets $\EE_{J_k}$ are all disjoint and $\mu$ has no mass outside $\bigcup_{k \in R} \EE_{J_k}$. 
Hence $G$ and $G_{\bM}$ are both mev distributions with unit-Fréchet margins and the same exponent measure. Thus $G_{\bM}=G$.
\end{proof}
\begin{proof}[Proof of Proposition~\ref{propo:densityofmixturemgp}]
The distribution of $\bU$ in \eqref{equa:generalgenerator} is a mixture of the distributions of the random vectors $\pi_{J_k}^{-1} \rbr{\bU^{(k)}+\rbr{\ln \rbr{a_{jk}/m_k} }_{j \in J_k} }$ over $k \in R$. The latter random vectors are absolutely continuous with respect to $\upsilon$ in \eqref{eq:upsilon} with density 
\[
f^{(k)}(\bu)=f_{\bU^{(k)}} \rbr{\bu_{J_k}- \rbr{\ln\rbr{a_{jk}/m_k}}_{j \in J_k}} \indicator \cbr{\bu \in \mathbb{A}_{J_k}}, \qquad \bu \in [-\infty,\infty)^d,
\]
where $\mathbb{A}_{J_k}$ as in \eqref{eq:AJ}. Hence $\bU$ is absolutely continuous with respect to $\upsilon$ with density 
\[
f_{\bU} \rbr{\bu}=\sum_{k \in R} m_{k} f_{\bU^{(k)}} \rbr{\bu_{J_k}- \rbr{\ln\rbr{a_{jk}/m_k}}_{j \in J_k} } \indicator \cbr{\bu \in \mathbb{A}_{J_k}}, \qquad  \bu \in [-\infty,\infty)^d.
\]
Finally, Theorem~\ref{theorem:pdfofUwrtmu} yields that the mgp random vector $\bY$ is absolutely continuous with respect to $\upsilon$ with density $h_{\bY}$ in \eqref{equa:densityofmixturemgp}.
\end{proof}

\begin{proof}[Proof of Proposition~\ref{propo:generatormixturelogisticmodel}]
For $k \in R$, the random vector $\bU^{(k)}$ in \eqref{equa:generatormixturemodel} satisfies \eqref{eq:generator} and \eqref{equa:conditiononU} for the logistic stdf with parameter $\alpha_k $ \citep[Proposition~1.2.1]{MF19}.  Finally, Proposition~\ref{propo-generator-general-model} finishes the proof.
\end{proof}

\begin{proof}[Proof of Proposition~\ref{propo:density}]
For $k \in R$, the random vector $\bU^{(k)}$ in \eqref{equa:generatormixturemodel} has independent Gumbel components. Thus $\bU^{(k)}$ is absolutely continuous with respect to $\lambda_{\lvert J_k \rvert}$. Proposition~\ref{propo:densityofmixturemgp} yields that the mgp random vector $\bY$ associated with the mixture logistic model is absolutely continuous with respect to $\upsilon$ in \eqref{eq:upsilon} with density $h_{\bY}$ given by \eqref{equa:densityofmixturemgp}. We use \citet[Subsection~7.2]{A16} to calculate the integrals in the latter equation. This yields the density of $\bY$ in \eqref{equa:densitymixturelogistic}.
\end{proof}

\begin{proof}[Proof of Proposition~\ref{propo:generatormixtureHRmodel}]
For $k \in R$, the random vector $\bU^{(k)}$ in \eqref{equa:generatormixtureHRmodel} satisfies \eqref{eq:generator} and \eqref{equa:conditiononU} for the \HR~stdf with variogram matrix $\Gamma^{(k)} $; see, e.g., \citet{wadsworth2014efficient}. Finally, Proposition~\ref{propo-generator-general-model} finishes the proof.
\end{proof}
\begin{proof}[Proof of Proposition~\ref{Propo:densityHR}]
For $k \in R$, the random vector $\bU^{(k)}$ in \eqref{equa:generatormixtureHRmodel} is a $\lvert J_k \rvert$-variate normal random vector with positive definite covariance matrix. Thus $\bU^{(k)}$ is absolutely continuous \wrto $\lambda_{\lvert J_k \rvert}$. Proposition~\ref{propo:densityofmixturemgp} yields that the mgp random vector $\bY$ associated with the mixture \HR~model is absolutely continuous with respect to $\upsilon$ with density $h_{\bY}$ given by \eqref{equa:densityofmixturemgp}. We use \citet[Section~7.2]{A16} to calculate the integrals in the latter equation. This yields the density of $\bY$ given by \eqref{equa:densitymixturehr}.
\end{proof}

\section{Proofs of Section~\ref{sec:simulation}}
\label{sec:proofs5}


\begin{proof}[Proof of Proposition~\ref{lemma:conditionalT's}]
For $k \in R$, let the random vector $\bW^{(k)} = \rbr{W^{(k)}_1,\ldots,W^{(k)}_d}$ with values in $\mathbb{A}_{J_k}$ be defined by
\[
	W_j^{(k)} =
	\begin{dcases}
		P^{(k)}_j & \text{if $j \in J_k$}, \\
		-\infty & \text{if $j \in D \setminus J_k$.}
	\end{dcases}
\]
Then Eq.~\eqref{equa:generalgenerator} is the same as
\begin{equation}
\label{eq:mkWk2U}
	\pr\sbr{\bU \in \point}
	=
	\sum_{k \in R} m_k \, \pr\sbr{ \bW^{(k)} \in \point }
\end{equation}
Further, we have, by the relation between the stdf $\ell^{(k)}$ and the generator $\bU^{(k)}$,
\begin{align*}
	\expec \sbr{ \exp \cbr{\max\rbr{ \bW^{(k)} }} }
	&= \expec \sbr{ \exp \cbr{ \max\rbr{\bP^{(k)}}} } \\
	&= \expec \sbr{ \max_{j \in J_k} \cbr{ \frac{a_{jk}}{m_k} \ee^{U_j^{(k)}} } }
	= \frac{1}{m_k} \ell^{(k)} \rbr{ (a_{jk})_{j \in J_k} }.
\end{align*}
Since also $\ell(\bone) = \expec \sbr{\ee^{\max(\bU)}}$, we find
\begin{equation}
\label{eq:mkWk2wk}
	\frac{m_k\expec\sbr{\ee^{\max \bW^{(k)} }}}{\expec\sbr{\ee^{\max \bU}}}
	= \frac{\ell^{(k)} \rbr{(a_{jk})_{j \in J_k}}}{\ell(\bone)} 
	= w_k	
\end{equation}
as defined in \eqref{equa:lawT^k}. By Equations~\eqref{equa:TtoU}, \eqref{eq:mkWk2U} and \eqref{eq:mkWk2wk}, we find, Borel sets $B \subseteq \left[-\infty, \infty\right)^d$,
\begin{align*}
 	\pr\sbr{\bT \in B}
 	&= \frac{\expec\sbr{\ee^{\max \bU} \indicator\cbr{\bU \in B}}}{\expec\sbr{\ee^{\max \bU}}} 
 	=\sum_{k \in R} \frac{m_k}{\expec\sbr{\ee^{\max \bU}}} \expec\sbr{\ee^{\max \bW^{(k)} } \indicator\cbr{ \bW^{(k)} \in B}} \\
    &= \sum_{k \in R} w_k \frac{\expec\sbr{\ee^{\max \bW^{(k)} } \indicator\cbr{ \bW^{(k)} \in B}}}{\expec\sbr{\ee^{\max \bW^{(k)} }}} \\
    &= \sum_{k \in R} w_k \frac{\expec\sbr{\ee^{\max \bP^{(k)} } \indicator\cbr{ \bP^{(k)} \in \pi_{J_k}\rbr{\mathbb{A}_{J_k} \cap B}}}}{\expec\sbr{\ee^{\max \bP^{(k)} }}} \\
    &= \sum_{k \in R} w_k  \pr\sbr{ \bT^{(k)} \in \pi_{J_k} \rbr{\mathbb{A}_{J_k} \cap B} },
\end{align*}
where $\bT^{(k)}$ is a random vector on $\reals^{J_k}$ which satisfies \eqref{equa:lawT^k}.

The expression for the density of $\bT^{(k)}$ follows from Lemma~\ref{lemma:TktoTK} below applied to $\bP = \bP^{(k)}$ upon noting that
\[
	\expec \sbr{ \exp \rbr{P_j^{(k)}} }
	= \expec \sbr{ \exp \cbr{U_j^{(k)} + \ln(a_{jk}/m_k)} }
	= a_{jk}/m_k, \qquad j \in J_k,
\]
as well as $\expec \sbr{ \exp \cbr{ \max \rbr{ \bP^{(k)} } } } = \ell^{(k)} \rbr{ (a_{jk})_{j \in J_k} } / m_k$, which was shown above.
\end{proof}

\begin{lemma}
	\label{lemma:TktoTK}
	Let $J$ be a positive integer. If $\bo{P}$ is a random vector on $\reals^J$ with Lebesgue density $f_{\bP}$ such that $p_j := \expec\sbr{\exp(P_j)} < \infty$ for all $j \in \{1,\ldots,J\}$, then the random vector $\bT$ with distribution $\pr\sbr{\bT \in \point} = \expec\sbr{\ee^{\max \bo{P}} \indicator\cbr{\bo{P} \in \point}} / \expec\sbr{\ee^{\max \bo{P}}}$ has Lebesgue density
	\begin{equation*}
		\label{eq:fV2fT} 
		f_{\bT}(\bt) = c \, g(\bt) \sum_{j=1}^J n_j \, q_j(\bt), \qquad \bt \in \reals^J,
	\end{equation*}
	with $c = \sum_{j=1}^J p_j / \expec\sbr{\ee^{\max \bo{P}}}$, $g(\bt) = \ee^{\max \bt} / \sum_{j=1}^J \ee^{t_j}$, mixture weights $n_j = p_j / \sum_{i=1}^J p_i$ and probability densities $q_j(\bt) = \ee^{t_j} f_{\bo{P}}(\bt) / p_j$.  
\end{lemma}

\begin{proof}[Proof of Lemma~\ref{lemma:TktoTK}]
The density of $\bT$ evaluated in $\bt \in \reals^J$ is
\begin{align*}
    f_{\bT}(\bt) 
    = \frac{\ee^{\max \bt}}{\expec\sbr{\ee^{\max \bP}}} f_{\bP}(\bt) 
    = \frac{\sum_{i=1}^J p_i}{\expec\sbr{\ee^{\max \bP}}} 
    \frac{\ee^{\max \bt}}{\sum_{j=1}^J \ee^{t_j}} 
    \sum_{j=1}^J \frac{p_j}{\sum_{i=1}^J p_i} \frac{\ee^{t_j}}{p_j} f_{\bP}(\bt),
\end{align*}
which is the expression stated in the lemma.
\end{proof}

\begin{proof}[Proof of Proposition~\ref{propo:extremalfunction}]
Equation~\eqref{09:29:14:29} yields that for measurable, nonnegative functions $g$, the distributions of $\bS$ and $\bU$ are related via
    \begin{equation*}
        \expec\sbr{g(\bS)} =
        \frac{\expec\sbr{g\rbr{\bU-\max \bU}\ee^{\max \bU}}}{\expec\sbr{\ee^{\max \bU}}}.
    \end{equation*}
Fix $\by \in [0,\infty)^{d}$ with $y_j \geq 1$.
For $g(\bo{s}) = \indicator \cbr{ \bo{s} - s_j \le \ln(\by) } \ee^{s_j}$, we get
\begin{align*}
	g(\bU - \max \bU)
	&= \indicator\cbr{ \rbr{\bU - \max \bU} - \rbr{U_j - \max \bU} \le \ln(\by)} \ee^{U_j - \max \bU} \\
	&= \indicator\cbr{ \bU - U_j \le \ln(\by) } \ee^{U_j - \max \bU}
\end{align*}
and thus
\begin{align*}
	\expec \sbr{\indicator\cbr{\bS-S_j \leq \ln(\by)}\ee^{S_j}}
	&=
	\frac{\expec\sbr{\indicator\cbr{ \bU - U_j \le \ln(\by) } \ee^{U_j - \max \bU} \ee^{\max \bU}}}%
	{\expec\sbr{\ee^{\max \bU}}} \\
	&= 	\frac{\expec\sbr{\indicator\cbr{ \bU - U_j \le \ln(\by) } \ee^{U_j}}}%
	{\expec\sbr{\ee^{\max \bU}}}.
\end{align*}
Further, we deduce from \eqref{equa:defstdf}, \eqref{11:21:01:39}, \eqref{eq:generator} and \eqref{equa:conditiononU} that
\[
	\pr\sbr{ Y_j > 0 } 
	= \frac{1}{\mu\rbr{\cbr{\bx \ge 0 : \; \bx \not\leq \bone}}} 
	= \frac{1}{\ell(\bone)} 
	= \frac{1}{\expec\sbr{\ee^{\max \bU}}}.
\]
Recall representation $\bY \eqd E+\boldsymbol{S}$ with $E$ a unit exponential random variable independent from $\bS$.
Since, by definition, $\pr\sbr{ \bQ^{(j)} \in \diff \bt } = \ee^{t_j} \pr\sbr{ \bU \in \diff \bt}$, we conclude that
\begin{align*}
	\pr\sbr{\ee^{\bQ^{(j)}-Q^{(j)}_j} \leq \by}
	&= \expec\sbr{\indicator\cbr{\ee^{\bU-U_j}\leq \by}\ee^{U_j}} \\
	&= \expec \sbr{\indicator\cbr{\bS-S_j \leq \ln(\by)}\ee^{S_j}} \cdot \expec\sbr{\ee^{\max \bU}} \\
	&= \pr\sbr{(E+\bS)-(E+S_j) \leq \ln(\by), \, E>-S_j} / \pr\sbr{Y_j > 0} \\
	&=\pr \sbr{\ee^{\bY}/\ee^{Y_j} \leq \by , \, Y_j > 0} / \pr\sbr{Y_j > 0} \\
	&=\pr \sbr{\bo{\mathcal{U}}^{(j)} \leq \by},
\end{align*}
where the last equality is proved in \citet[Lemma~2]{engelke2020graphical}.
\end{proof}

\begin{proof}[Proof of Proposition~\ref{propo:simulationqlogostic}]
Fix $j \in J_k$. If $i \in J_k \setminus \cbr{j}$, then the density of $\psi_i$ is clearly $t \mapsto f_{U^{(k)}_i}\rbr{t-\ln \rbr{a_{ik}/m_k}}$. Otherwise if $i=j$, by straightforward calculus, we get 
\[
\frac{\partial}{\partial t} \pr \sbr{ -\alpha_k\ln N+\ln\rbr{a_{jk}/\rbr{m_k \Gamma(1-\alpha_k)} } \leq t }=\frac{\ee^{t} f_{U_{j}^{(k)}}\rbr{t-\ln \rbr{a_{jk}/m_k}}}{\bigintss_{s \in \reals} \ee^{s} f_{U_{j}^{(k)}}\rbr{s-\ln \rbr{a_{jk}/m_k}} \diff s}.
\qedhere
\]
\end{proof}

\begin{proof}[Proof of Proposition~\ref{propo:pdfhrq}]
The proof is by direct calculations and is omitted for brevity.
\end{proof}



\paragraph{Funding} The research of Anas Mourahib was supported financially by the \emph{Fonds de la Recherche Scientifique – FNRS}, Belgium (grant
number T.0203.21).


\paragraph{Acknowledgments} We are grateful to the anonymous reviewers for their suggestions which greatly helped to improve the manuscript.

\bibliographystyle{plainnat}
\bibliography{biblio}

\end{document}

%% file: newcommands.tex
\newcommand{\GP}{\operatorname{GP}}  
\newcommand{\Gumbel}{\operatorname{Gumbel}}

\newcommand{\expec}{\operatorname{\mathsf{E}}}

\newcommand{\diag}{\operatorname{diag}}
\newcommand{\indicator}{\operatorname{\mathbbm{1}}}
\newcommand{\DA}{\operatorname{DA}}

\newcommand{\GPU}{\GP_{\bU}}

\renewcommand{\ge}{\geqslant}
\renewcommand{\geq}{\geqslant}
\renewcommand{\le}{\leqslant}
\renewcommand{\leq}{\leqslant}

\newcommand{\bQ}{\boldsymbol{Q}}
\newcommand{\bS}{\boldsymbol{S}}
\newcommand{\bt}{\boldsymbol{t}}
\newcommand{\bu}{\boldsymbol{u}}
\newcommand{\bU}{\boldsymbol{U}}
\newcommand{\bV}{\boldsymbol{V}}
\newcommand{\bP}{\boldsymbol{P}}
\newcommand{\bv}{\boldsymbol{v}}
\newcommand{\bw}{\boldsymbol{w}}
\newcommand{\bW}{\boldsymbol{W}}
\newcommand{\bx}{\boldsymbol{x}}
\newcommand{\bX}{\boldsymbol{X}}
\newcommand{\bM}{\boldsymbol{M}}
\newcommand{\by}{\boldsymbol{y}}
\newcommand{\bY}{\boldsymbol{Y}}
\newcommand{\bz}{\boldsymbol{z}}
\newcommand{\bZ}{\boldsymbol{Z}}

\newcommand{\bN}{\boldsymbol{N}}

\newcommand{\bo}[1]{\boldsymbol{#1}} 
\newcommand{\binfty}{\boldsymbol{\infty}}
\newcommand{\bzero}{\boldsymbol{0}}
\newcommand{\bone}{\boldsymbol{1}}
\newcommand{\bgamma}{\boldsymbol{\gamma}}
\newcommand{\bT}{\boldsymbol{T}}
\newcommand{\bsigma}{\boldsymbol{\sigma}}

\newcommand{\boeta}{\bo{\eta}}

\newcommand{\TJ}{\bT_{\mid J}}

\newcommand{\YJ}{\bY_{\mid J}}
\newcommand{\UJ}{\bU_{\mid J}}
\newcommand{\ZJ}{\bZ_{\mid J}}

\newcommand{\UmixA}{\bV}
\newcommand{\UmixlogA}{\bU}

\newcommand{\bxi}{\bo{\xi}}

\newcommand{\Mb}{\mathcal{M}_{b}}

\newcommand{\borel}{\mathcal{B}}
\newcommand{\sphere}{\mathbb{S}} 
\newcommand{\mBI}{I_J}

\newcommand{\EE}{\mathbb{E}}

\newcommand{\reals}{\mathbb{R}}
\newcommand{\Rd}{\reals^d}

\newcommand{\rectangle}{\mathbb{R}_{n,J}^{\eps}}
\newcommand{\halfcone}{\mathbb{C}_{n,J}^{\epsilon}}
\newcommand{\rectangleinfty}{\mathbb{R}_{n,J}^{\eps,\infty}}
\newcommand{\halfconeinfty}{\mathbb{C}_{n,J}^{\epsilon,\infty}}
\newcommand{\EEJ}{\EE_{J}}

\newcommand{\sphereJ}{\sphere_{J}}

\newcommand{\scenario}{\mathcal{R}}

\DeclarePairedDelimiterX{\norma}[1]{\lVert}{\rVert}{#1}
\newcommand{\bnorm}[1]{\| {#1} \|} 
\newcommand{\mass}{\pi_j}
\newcommand{\ka}{k}

\newcommand{\bigar}{R}

\newcommand{\norm}{\|\,\cdot\,\|}

\newcommand{\point}{\,\cdot\,}
\newcommand{\Normal}{\mathcal{N}}
\newcommand{\ee}{\mathrm{e}} 
\newcommand{\pr}{\operatorname{\mathsf{P}}}
\newcommand{\diff}{\, \mathrm{d}} 
\newcommand{\eqd}{\overset{\mathsf{d}}{=}} 
\newcommand{\dto}{\rightsquigarrow}

\newcommand{\eps}{\epsilon}
\newcommand{\law}{\mathcal{L}}

\newcommand{\Rtwo}{R_{\geq 2}}
\newcommand{\Rone}{R_{= 1}}
\newcommand{\cbr}[1]{\left\{ {#1} \right\}} 
\newcommand{\rbr}[1]{\left( {#1} \right)}   
\newcommand{\sbr}[1]{\left[ {#1} \right]}   

\newcommand{\HR}{Hüsler--Reiss}
\newcommand{\FR}{Fr\'echet}
\newcommand{\extreme}{extreme direction}
\newcommand{\extremes}{extreme directions}
\newcommand{\wrto}{w.r.t.\ }

\definecolor{navyblue}{rgb}{0.0, 0.0, 0.5}
\definecolor{darkcerulean}{rgb}{0.03, 0.27, 0.49}
\definecolor{burntorange}{rgb}{0.8, 0.33, 0.0}
\definecolor{britishracinggreen}{rgb}{0.0, 0.26, 0.15}
\definecolor{alizarin}{rgb}{0.82, 0.1, 0.26}
\newcommand{\js}[1]{}

\newcommand{\jsch}[1]{\textcolor{black}{#1}}
\newcommand{\amch}[1]{\textcolor{black}{#1}}
\newcommand{\ak}[1]{}

\definecolor{mediumpersianblue}{rgb}{0.0, 0.53, 0.85}
\newcommand{\amtwo}[1]{#1}

